\documentclass[ssy,preprint]{imsart}

\usepackage{amssymb,amsmath,amsthm,amsfonts,dsfont}
\usepackage{graphicx}
\usepackage{tikz}
\usetikzlibrary{decorations.pathreplacing}
\usepackage{tikz-cd}
\usepackage{hyperref}
\usepackage[nodayofweek]{datetime}
\usepackage{mathtools}

\mathtoolsset{showonlyrefs=true}

\startlocaldefs
\newdateformat{mydate}{\shortmonthname[\THEMONTH]{ }\twodigit{\THEDAY}, \THEYEAR}

\numberwithin{equation}{section}

\newtheorem{theorem}{Theorem}[section]
\newtheorem{lemma}[theorem]{Lemma}
\newtheorem{claim}[theorem]{Claim}
\newtheorem{proposition}[theorem]{Proposition}

\theoremstyle{definition}

\newtheorem{assumption}{Assumption}[section]
\newtheorem{remark}{Remark}[section]

\usepackage[normalem]{ulem}

\endlocaldefs

\begin{document}

\begin{frontmatter}
\title{A lower bound on the queueing delay in resource constrained load balancing\thanksref{T1}\thanksref{T2}}
\thankstext{T1}{Research supported in part by an MIT Jacobs Presidential Fellowship, the NSF grant CMMI-1234062, and the ONR grant N00014-17-1-2790.}
\thankstext{T2}{A preliminary version of this paper appeared in the Proceedings of the 2016 ACM Sigmetrics Conference \cite{sigmetrics16}.}

\begin{aug}
\author{\fnms{David} \snm{Gamarnik}\ead[label=e1]{gamarnik@mit.edu}},
\author{\fnms{John N.} \snm{Tsitsiklis}\ead[label=e2]{jnt@mit.edu}}
\and
\author{\fnms{Martin} \snm{Zubeldia}\ead[label=e3]{zubeldia@mit.edu}}


\affiliation{Massachusetts Institute of Technology}

%

\end{aug}


\begin{abstract}
We consider the following distributed service model: jobs with unit mean, general distribution, and independent processing times arrive as a renewal process of rate $\lambda n$, with $0<\lambda<1$, and are immediately dispatched to one of several queues associated with $n$ identical servers with unit processing rate. We assume that the dispatching decisions are made by a central dispatcher endowed with a  finite memory, and with the ability to exchange messages with the servers.

We study the fundamental resource requirements (memory bits and message exchange rate), in order to drive the expected queueing delay in steady-state of a typical job to zero, as $n$ increases. We develop a novel approach to show that, within a certain broad class of ``symmetric'' policies, every dispatching policy with a message rate of the order of $n$, and with a memory of the order of $\log n$ bits, results in an expected queueing delay which is bounded away from zero, uniformly as $n\to\infty$.
\end{abstract}

%

\end{frontmatter}

\setcounter{tocdepth}{2}
\tableofcontents

\section{Introduction}
Distributed processing systems are ubiquitous, from passport control at the airport and checkout lines at the supermarket, to call centers and server farms for cloud computing. Many of these systems involve a stream of incoming jobs dispatched to
distributed queues, with each queue associated to a different server;  see Figure \ref{fig:basicSetting} for a stylized model.
Naturally, the performance of such systems depends critically on the policy used to dispatch jobs to queues.

 \begin{figure}[ht!]
 \centering
 \begin{tikzpicture}[scale=0.7]
    \draw [very thick,->] (-0.5,0) -> (1.5,0);
    \draw (0.5,0.5) node {Incoming jobs};
     \draw [thick] (2,-1) -- (2,1) -- (4.5,1) -- (4.5,-1) -- (2,-1);
     \draw (3.25,0) node {Dispatcher};
     \draw [thick, ->] (4.5,0.7) -- (6.5,2.8);
     \draw [thick, ->] (4.5,-0.7) -- (6.5,-2.8);
     \draw [thick, ->] (4.5,0.4) -- (6.5,1.4);
     \draw [thick, ->] (4.5,-0.4) -- (6.5,-1.4);
     \begin{scope}[yshift=3cm,xshift=5cm]
          \draw [thick] (2,0.5) -- (3.5,0.5) -- (3.5,-0.5) -- (2,-0.5);
          \draw (4.35,0) circle (5mm);
     \end{scope}
     \begin{scope}[yshift=1.5cm,xshift=5cm]
          \draw [thick] (2,0.5) -- (3.5,0.5) -- (3.5,-0.5) -- (2,-0.5);
          \draw (4.35,0) circle (5mm);
     \end{scope}
     \begin{scope}[xshift=5cm]
          \draw (2.75,0) node {\Huge\vdots};
     \end{scope}
     \begin{scope}[yshift=-3cm,xshift=5cm]
          \draw [thick] (2,0.5) -- (3.5,0.5) -- (3.5,-0.5) -- (2,-0.5);
          \draw (4.35,0) circle (5mm);
     \end{scope}
     \begin{scope}[yshift=-1.5cm,xshift=5cm]
          \draw [thick] (2,0.5) -- (3.5,0.5) -- (3.5,-0.5) -- (2,-0.5);
          \draw (4.35,0) circle (5mm);
     \end{scope}
          \draw [thick,decorate,decoration={brace,amplitude=10pt}]   (10,3.5) -- (10,-3.5) node [midway,right,xshift=.5cm] {Servers};
  \end{tikzpicture}
  \caption{Parallel server queueing system with a central dispatcher.}
    \label{fig:basicSetting}
  \end{figure}
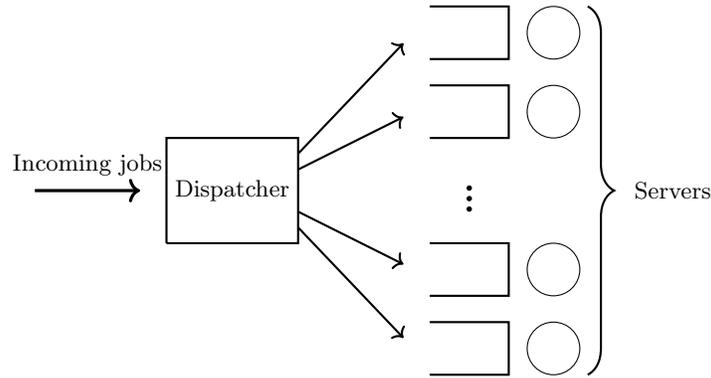

In order to take full advantage of the multiple servers, the dispatcher can benefit from information about the current state of the queues (e.g., whether they are empty or not).
For such information to be available when a job arrives to the system and a dispatching decision is to be made, it is necessary that the dispatcher periodically obtain information about the current queue states from the corresponding servers  and/or have sufficient memory that allows it to extrapolate from the available information. In this paper, we explore the tradeoff between the performance of the system and the amount of resources used to gather and maintain relevant information.

There is a variety of ways in which the system described above can be operated; these correspond to different dispatching policies, and result in different performance and resource utilization. At one extreme, the dispatcher can route incoming jobs to a queue chosen uniformly at random. This policy requires no information on the state of the queues, but the jobs experience considerable delays. At the other extreme, the dispatcher can send incoming jobs to a queue with the smallest number of jobs or to a queue with the smallest workload. The jobs in these last two policies experience little or no delay, but  substantial communication overhead between the dispatcher and the servers is required.

Many intermediate policies have been proposed and analyzed in the past (e.g., the power-of-$d$-choices \cite{mitzenmacher,vvedenskaya} or join-an-idle-queue \cite{badonnelBurgess,joinIdleQueue,stolyar14}), and they use varying amounts of resources to achieve different levels of delay performance. A more detailed discussion of the relevant literature and the various schemes therein is deferred to Section \ref{sec:review}.

Instead of focusing on yet another policy or decision making architecture, we step back and address a more fundamental question: what is the minimum amount of resources required to obtain the best possible performance, as the number of server increases? Regarding performance, we focus on the expected time that a job has to wait before starting service. Regarding resources, we focus on the average number of messages exchanged between the dispatcher and the servers per unit of time, and on the number of bits of ``long term" memory that the dispatcher has at its disposal.

Some performance-resources tradeoffs similar to the ones we study have been analyzed in the context of the balls into bins model \cite{ballsAndBins}, in which $n$ balls are to be placed sequentially into $n$ bins. In particular, the tradeoff between number of messages and maximum load was recently characterized by Lenzen and Wattenhofer \cite{tightBounds}. Furthermore, the tradeoff between memory size and maximum load was studied in \cite{choiceMemoryTradeoff,memoryPerformanceTradeoffs}.

\subsection{Our contribution}
We consider a broad family of decision making architectures and policies, which includes most of those considered in the earlier literature, and work towards characterizing the attainable delay performance for a given level of resources. We allow the dispatcher to have a limited memory, where it can store information on the state of the queues. We also allow the exchange of messages, either from the dispatcher to the servers (queries), or from the servers to the dispatcher (responses to queries or spontaneous status updates).

We show that if the average message rate is at most of the order of the arrival rate, and the memory size (in bits) is at most of the order of the logarithm of the number of servers, then \emph{every} decision making architecture and policy, within a certain broad class of dispatching policies, results in a queueing delay that does {\bf not} vanish as the system size increases. In particular, we show that the expected queueing delay in steady-state of a typical job is uniformly bounded away from zero as the number of servers goes to infinity. The main constraints that we impose on the policies that we consider are: (i) there is no queueing at the dispatcher, i.e., each job is immediately dispatched to one of the parallel queues, and (ii) policies are symmetric, in a sense to be made precise later.

\begin{remark}
 For the case of Poisson arrivals and exponential service times, in \cite{positiveResult}, the authors have shown that if either (i) the message rate is superlinear in the arrival rate and the memory is logarithmic in the number of servers, {\bf or} (ii) the memory size is superlogarithmic in the number of servers and the message rate is greater than or equal to the arrival rate, then there exists a dispatching policy that results in a vanishing queueing delay as the system size increases. This gives sufficient conditions to achieve a vanishing queueing delay that are complementary to the necessary conditions obtained in this paper.
\end{remark}

\subsection{Outline of the paper}
The remainder of the paper is organized as follows. In Section \ref{sec:notation} we introduce some notation. The model and the main result are presented in Section \ref{sec:results}. In Section \ref{sec:review} we discuss our result in the context of some concrete dispatching policies from the earlier literature. In Section \ref{sec:proof_imposs} we provide the proof of our main result. Finally, in Section \ref{sec:conclusions} we present our conclusions and suggestions for future work.

\section{Notation}\label{sec:notation}
In this section we introduce some notation that will be used throughout the paper.
For any positive functions $f$ and $g$, we write
$$  f(n)\in \omega(g(n)) \,\,\ \ \text{if and only if}\,\,\ \  \liminf\limits_{n\to \infty} \frac{f(n)}{g(n)} = \infty.
$$
We let $[\,\cdot\,]^+\triangleq \max\{\,\cdot\,,0\}$. We let $\mathbb{Z}_+$ and $\mathbb{R}_+$ be the sets of nonnegative integers and real numbers, respectively. The indicator function is denoted by $\mathds{1}$, so that $\mathds{1}_A(x)$ is $1$ if $x\in A$, and is $0$ otherwise. Given a set $A$, its power set, the set of all subsets of $A$, is denoted by $\mathcal{P}(A)$.
Random variables will always be denoted by upper case symbols. Non-random quantities will generally --- but not always --- be denoted by lower case symbols; exceptions will be pointed out as necessary.

We will use boldface fonts to denote vectors. If $\bf v$ is a vector, we denote its $i$-th component by ${\bf v}_i$. We will denote the (unordered) set of elements of a vector by using the superscript ``set''; for example, if ${\bf v}=(2,1,3,1)$, then ${\bf v}^{set}=\{1,2,3\}$. Furthermore, we will use $|{\bf v}|$ to denote the {\bf dimension} of a vector $\bf v$. If ${\bf v}=({\bf v}_1,\ldots,{\bf v}_m)$ is a vector, and $\bf u$ is a vector with entries in $\{1,\ldots,m\}$, then $\bf v_u$ is a $|{\bf u}|$-dimensional vector whose $i$-th component is ${\bf v}_{{\bf u}_i}$; for example, if ${\bf u}=(3,1)$, then ${\bf v_u}=({\bf v}_3,{\bf v}_1)$.

For any positive integer $n$, we define the sets $\mathcal{N}_n\triangleq\{1,\dots,n\}$, and
\begin{equation}\label{eq:defSn}
 \mathcal{S}_n\triangleq \left\{ {\bf s} \in \bigcup_{i=0}^n (\mathcal{N}_n)^i: \text{there are no repeated elements in }{\bf s} \right\},
\end{equation}
where $(\mathcal{N}_n)^0=\{\emptyset\}$. We say that a permutation $\sigma:\mathcal{N}_n\to\mathcal{N}_n$ {\bf fixes a set} $R\subset\mathcal{N}_n$ if $\sigma(i)=i$, for all $i\in R$.  Furthermore, we say that a permutation $\sigma$ {\bf preserves the ordering} of a subset $A\subset\mathcal{N}_n$ if $\sigma(i)<\sigma(j)$ whenever $i,j\in A$ and $i<j$. We use $P_k$ to denote the set of permutations of the set $\{1,\dots,k\}$. If ${\bf v}=({\bf v}_1,\ldots,{\bf v}_m)$ is a vector in $(\mathcal{N}_n)^m$ and $\sigma\in P_n$ is a permutation of $\mathcal{N}_n$, we denote by $\sigma({\bf v})$ the vector  $(\sigma({\bf v}_1),\ldots,\sigma({\bf v}_m))$. Finally, for any function $X(\cdot)$ of time, and any $t\in\mathbb{R}$, we let $X(t^-)=\lim_{\tau\uparrow t}X(\tau)$, as long as the limit exists.

\section{Model and main results}\label{sec:results}
In this section we present our main result. We first present a unified framework that defines a broad set of dispatching policies, which includes most of the policies studied in previous literature. We then present our negative result on the expected queueing delay under resource constrained policies within this set of policies.

\subsection{Modeling assumptions}\label{sec:model_assumption}
We consider a system consisting of $n$ parallel servers, where each server has a processing rate equal to $1$. Furthermore, each server is associated with an infinite capacity FIFO queue. Jobs arrive to the system as a single renewal process of rate $\lambda n$, for some fixed $\lambda<1$. Job sizes are i.i.d., independent from the arrival process, and have an arbitrary distribution with mean $1$. We use the convention that a job that is being served remains in queue until its processing is completed. We assume that each server is work-conserving: a server is idle if and only if the corresponding queue is empty.

A central controller (dispatcher) is responsible for routing each incoming job to a  queue, immediately upon arrival. The dispatcher has limited information on the state of the queues; it can only rely on a limited amount of local memory and on messages that provide partial information about the state of the system. These messages (which are assumed to be instantaneous) can be sent from a server to the dispatcher at any time, or from the dispatcher to a server (in the form of queries) at the time of an arrival. Messages from a server can only contain information about the state of its own queue (number of remaining jobs and the remaining workload of each one). Within this context, a system designer has the freedom to choose a messaging policy, as well as the rules for updating the memory and for selecting the destination of an incoming job.

We are interested in the case where $n$ is very large, in the presence of constraints on the rate of message exchanges and on the memory size. The performance metric that we focus on is the {\bf expected queueing delay in steady-state of a typical job}, i.e., the expected time between its arrival and the time at which it starts receiving service. We will formalize this definition in Section \ref{sec:imposs}.

\subsection{Unified framework for dispatching policies} \label{sec:unifiedFramework}
In this subsection we present a unified framework that describes memory-based dispatching policies. In order to do this, we introduce a sample path construction of the evolution of the system under an arbitrary policy.

Let $c_n$ be the number of memory bits available to the dispatcher. We define the corresponding set of memory states to be $\mathcal{M}_n \triangleq \left\{1,\dots,2^{c_n}\right\}$. 
Furthermore, we define the set of possible states at a server as the set of nonnegative sequences $\mathcal{Q}\triangleq\mathbb{R}_+^{\mathbb{Z}_+}$, where a sequence specifies the remaining workload of each job in that queue, including the one that is being served. (In particular, an idle server is represented by the zero sequence.) As long as a queue has a finite number of jobs, the queue state is a sequence that has only a finite number of non-zero entries. The reason that we include the workload of the jobs in the state is that we wish to allow for a broad class of policies, that can take into account the remaining workload in the queues.
In particular, we allow for information-rich messages that describe the full workload sequence at the server that sends the message. We are interested in the process
\[ {\bf Q}(t)= \big({\bf Q}_1(t),\dots,{\bf Q}_n(t)\big) = \Big(\big({\bf Q}_{1,j}(t)\big)_{j=1}^\infty,\dots,\big({\bf Q}_{n,j}(t)\big)_{j=1}^\infty\Big)\in \mathcal{Q}^n,\]
which describes the evolution of the workload of each job in each queue. Here ${\bf Q}_{i,j}(t)$ is the remaining workload of the $j$-th job in the $i$-th queue, at time $t$, which for $j\geq 2$ is simply the job's service time. We are also interested in the process $M(t)\in \mathcal{M}_n$ that describes the evolution of the memory state, and in the process $Z(t)\in\mathbb{R}_+$ that describes the remaining time until the next arrival of a job.

\subsubsection{Fundamental processes and initial conditions}\label{s:fund}
The processes of interest will be driven by certain common fundamental processes.
\begin{enumerate}
  \item {\bf Arrival process:} A delayed renewal counting process $A_n(t)$ with rate $\lambda n$, and event times $\{T_k\}_{k=1}^\infty$, defined on a probability space $(\Omega_A,\mathcal{A}_A,\mathbb{P}_A)$.
  \item {\bf Spontaneous messages process:} A Poisson counting process $R_n(t)$ with rate $\mu n$, and event times $\{T^s_k\}_{k=1}^\infty$, defined on a probability space $(\Omega_R,\mathcal{A}_R,\mathbb{P}_R)$.
  \item {\bf Job sizes:} A sequence of i.i.d. random variables $\{W_k\}_{k=1}^\infty$ with mean one, defined on a probability space $(\Omega_W,\mathcal{A}_W,\mathbb{P}_W)$.
  \item {\bf Randomization variables:} Four independent and individually i.i.d. sequences of random variables $\{U_k\}_{k=1}^\infty$, $\{V_k\}_{k=1}^\infty$, $\{X_k\}_{k=1}^\infty$, and $\{Y_k\}_{k=1}^\infty$,  uniform on $[0,1]$, defined on a common probability space $(\Omega_U,\mathcal{A}_U,\mathbb{P}_U)$.
  \item {\bf Initial conditions:} Random variables ${\bf Q}(0)$, $M(0)$, and $Z(0)$, defined on a common probability space $(\Omega_0,\mathcal{A}_0,\mathbb{P}_0)$.
\end{enumerate}
The whole system will be defined on the associated product probability space
\[ \big(\Omega_A\times\Omega_R\times\Omega_W\times\Omega_U\times\Omega_0, \mathcal{A}_A\times\mathcal{A}_R\times\mathcal{A}_W\times\mathcal{A}_U\times\mathcal{A}_0, \mathbb{P}_A\times\mathbb{P}_R\times\mathbb{P}_W\times\mathbb{P}_U\times\mathbb{P}_0\big), \]
to be denoted by $(\Omega,\mathcal{A},\mathbb{P})$. All of the randomness in the system originates from these fundamental processes, and everything else is a deterministic function of them.

\subsubsection{A construction of sample paths}\label{sec:samplePaths}
We consider some fixed $n$, and provide a construction of a Markov process $({\bf Q}(t),M(t),Z(t))$, that takes values in the set $\mathcal{Q}^n\times\mathcal{M}_n\times\mathbb{R}_+$. The memory process $M(t)$ is piecewise constant, and can only jump at the time of an event. All processes to be considered will have the c\`adl\`ag property (right-continuous with left limits) either by assumption (e.g., the underlying fundamental processes) or by construction.

There are three types of events: job arrivals, spontaneous messages, and service completions. We now describe the sources of these events, and what happens when they occur.\\

\noindent{\bf Job arrivals:} At the time of the $k$-th event of the arrival process $A_n$, which occurs at time $T_k$ and involves a job with size $W_k$, the following transitions happen sequentially but instantaneously.
\begin{enumerate}
  \item First, the dispatcher chooses a vector of distinct servers ${\bf S}_k$, from which it solicits information about their state, according to
  \[ {\bf S}_k=f_1\Big(M\left(T_k^-\right),W_k,U_k\Big), \]
  where $f_1:\mathcal{M}_n\times\mathbb{R}_+\times[0,1]\to\mathcal{S}_n$ is a measurable function defined by the policy. Note that the set of servers that are sampled only depends on the current memory state and on the size of the incoming job, but it is chosen in a randomized way, thanks to the independent random variable $U_k$. Thus, we allow for randomized policies; for example, the dispatcher might choose to sample a fixed number of servers uniformly at random.
  \item Then, messages are sent to the servers in the vector ${\bf S}_k$, and the servers respond with messages containing their queue states; thus, the information received by the dispatcher is the vector ${\bf Q}_{{\bf S}_k}$. This results in $2|{\bf S}_k|$ messages exchanged. Using this information, the destination of the incoming job is chosen to be
  \[ D_k=f_2\Big(M\big(T_k^-\big),W_k,{\bf S}_k,{\bf Q}_{{\bf S}_k}\big(T_k^-\big),V_k\Big), \]
  where $f_2:\mathcal{M}_n\times\mathbb{R}_+\times\mathcal{S}_n\times \big(\cup_{i=0}^n\mathcal{Q}^i\big)\times[0,1]\to\mathcal{N}_n$ is a measurable function defined by the policy. Note that the destination of a job can only depend on the current memory state, the job size, as well as the vector of queried servers and the state of their queues, but it is chosen in a randomized way, thanks to the independent random variable $V_k$. Once again, we allow for randomized policies that, for example, dispatch jobs uniformly at random.
  \item Finally, the memory state is updated according to
  \[ M(T_k)=f_3\Big(M\big(T_k^-\big),W_k,{\bf S}_k,{\bf Q}_{{\bf S}_k}\big(T_k^-\big),D_k\Big), \]
  where $f_3:\mathcal{M}_n\times\mathbb{R}_+\times\mathcal{S}_n\times \big(\cup_{i=0}^n\mathcal{Q}^i\big)\times\mathcal{N}_n\to\mathcal{M}_n$ is a measurable function defined by the policy. Note that the new memory state is obtained using the same information as for selecting the destination, plus the destination of the job, but without randomization.
\end{enumerate}

\noindent{\bf Spontaneous messages:} At the time of the $k$-th event of the spontaneous message process $R_n$, which occurs at time $T^s_k$, the $i$-th server sends a spontaneous message to the dispatcher if and only if
\[ g_1\Big({\bf Q}\big(T^s_k\big),X_k\Big)=i, \]
where $g_1:\mathcal{Q}^n\times[0,1]\to\{0\}\cup\mathcal{N}_n$ is a measurable function defined by the policy. In that case, the memory is updated to the new memory state
\[ M(T^s_k) = g_2\Big(M\big({T^s_k}^-\big), i, {\bf Q}_i\big(T^s_k\big)\Big), \]
where $g_2: \mathcal{M}_n\times \mathcal{N}_n \times \mathcal{Q}
\to\mathcal{M}_n$ is a measurable function defined by the policy, and which prescribes the server who sends a message. On the other hand, no message is sent when $g_1\big({\bf Q}({T^s_k}),X_k\big)=0$. Note that the dependence of $g_1$ on ${\bf Q}$ allows the message rate at each server to depend on the server's current workload. For example, we could let idle servers send repeated spontaneous messages (as a Poisson process) to inform the dispatcher of their idleness.\\

\noindent{\bf Service completions:} As time progresses, the remaining workload of each job that is at the head of line in a queue decreases at a constant, unit rate. When a job's workload reaches zero, the job leaves the system and every other job advances one slot. Let $\{T^d_k(i)\}_{k=1}^\infty$ be the sequence of departure times at the $i$-th server. At those times, the $i$-th server sends a message to the dispatcher if and only if
\[ h_1\Big({\bf Q}_i\big({T^d_k(i)}\big),Y_k\Big)=1, \]
where $h_1:\mathcal{Q}\times[0,1]\to\{0,1\}$ is a measurable function defined by the policy. In that case, the memory is updated to the new memory state
\[ M\Big(T^d_k(i)\Big) = h_2\Big(M\big({T^d_k(i)}^-\big), i, {\bf Q}_i\big(T^d_k(i)\big)\Big), \]
where $h_2:\mathcal{M}_n\times\mathcal{N}_n\times\mathcal{Q}\to\mathcal{M}_n$ is a measurable function defined by the policy. On the other hand, no message is sent when $h_1\big({\bf Q}_i({T^d_k(i)}),Y_k\big)=0$.

\begin{remark}
We have chosen to describe the collection of queried servers by a vector, implying an ordering of the servers in that collection. We could have described this collection as an (unordered) set.
These two options are essentially equivalent but
it turns out that the ordering provided by the vector description allows for a simpler presentation of the proof.
\end{remark}

\begin{remark}
  For any given $n$, a policy is completely determined by the spontaneous message rate $\mu$, and the functions $f_1$, $f_2$, $f_3$, $g_1$, $g_2$, $h_1$, and $h_2$. Furthermore, many policies in the literature that are described without explicit mention of memory or messages can be cast within our framework, as we will see in Section \ref{sec:review}.
\end{remark}

\begin{remark}
The memory update functions $f_3$, $g_2$, and $h_2$ do not involve randomization, even though our main result could be extended in that direction. We made this choice because none of the policies introduced in earlier literature require such randomization, and because it simplifies notation and the proofs.
\end{remark}

\begin{remark}
We only consider the memory used to store information in between arrivals or messages. Thus, when counting the memory resources used by a policy, we do not take into account information that is used in zero time (e.g., the responses from the queries at the time of an arrival) or the memory required to evaluate the various functions that describe the policy. If that additional memory were to be accounted for, then any memory constraints would be more severe, and therefore our negative result would still hold.
\end{remark}

The dispatching policies that we have introduced  obey certain constraints:
\begin{itemize}
  \item [(i)] The dispatcher can only send messages to the servers at the time of an arrival, and in a single round of communication. This eliminates the possibility of policies that sequentially poll the servers uniformly at random until they find an idle one. Indeed, it can be shown that such sequential polling policies may lead to asymptotically vanishing delays, without contradicting our lower bounds. On the other hand, in practice, queries involve some processing and travel time $\epsilon$.  Were we to consider a more realistic model with $\epsilon>0$, sequential polling would also incur positive delay.
  \item [(ii)] We assume that the dispatcher must immediately send an incoming job to a server upon arrival. This prevents the dispatcher from maintaining a centralized queue and operating the system as a G/G/$n$ queue.
\end{itemize}

We now introduce a symmetry assumption on the policies. In essence it states that at the time of a job arrival, and given the current memory state,  if certain sampling and dispatching decisions and a certain memory update are possible,  then a permuted version of these decisions and updates is also possible (and equally likely), starting with a suitably permuted memory state.

\begin{assumption}(Symmetric policies.)\label{def:symmetry}
We assume that the dispatching policy is symmetric, in the following sense. For any given permutation of the servers $\sigma$, there exists a corresponding (not necessarily unique) permutation $\sigma_M$ of the memory states $\mathcal{M}_n$ that satisfies all of the following properties.
\begin{enumerate}
  \item For every $m\in\mathcal{M}_n$ and $w \in \mathbb{R}_+$, and if $U$ is a uniform random variable on $[0,1]$, then
      \[ \sigma\Big(f_1(m,w,U)\Big) \overset{d}{=} f_1\big(\sigma_M(m),w,U\big), \]
      where $\overset{d}{=}$ stands for equality in distribution.
  \item For every $m\in\mathcal{M}_n$, $w \in \mathbb{R}_+$, ${\bf s}\in\mathcal{S}_n$, and ${\bf q}\in\mathcal{Q}^{|{\bf s}|}$, and if $V$ is a uniform random variable on $[0,1]$, then\footnote{Note that the argument on the right-hand side of the relation below involves
  ${\bf q}$ rather than a permuted version of ${\bf q}$, even though the vector ${\bf s}$ gets permuted. We are essentially comparing a situation where the dispatcher queries a vector ${\bf s}$ and receives certain numerical values ${\bf q}$ with the situation where the dispatcher queries a vector $\sigma({\bf s})$ and receives the {\bf same} numerical values ${\bf q}$.}
   \begin{align*}
   &\sigma\Big(f_2\big(m,w,{{\bf s},{\bf q}},V\big)\Big) \overset{d}{=} f_2\big(\sigma_M(m),w,{\sigma({\bf s}),{\bf q}},V\big).
  \end{align*}
  \item For every $m\in\mathcal{M}_n$, $w \in \mathbb{R}_+$, {${\bf s}\in\mathcal{S}_n$, and ${\bf q}\in\mathcal{Q}^{|{\bf s}|}$}, and $d\in\mathcal{N}_n$, we have
  \begin{align*}
   &\sigma_M\Big(f_3\big(m,w,{{\bf s},{\bf q}},d\big)\Big)= f_3\big(\sigma_M(m),w,{\sigma({\bf s}),{\bf q}},\sigma(d)\big).
  \end{align*}
\end{enumerate}
\end{assumption}

As a concrete illustration, our symmetry assumption implies the following. If a certain memory state mandates that the vector $(2,4,5)$ of servers must be sampled (with probability 1), independently from the incoming job size, then there exists some other memory state which mandates that the vector $(1,5,7)$ will be sampled, independently from the incoming job size, and the same holds for every 3-element vector with distinct entries. Since there are $n(n-1)(n-2)$ different vectors,
 there must be at least so many different memory states. This suggests that if we have too few memory states, the number of ``distinguished'' servers, i.e., servers that are treated in a special manner is severely limited. This is a key element of the proof of the delay lower bound that we present in the next subsection.

\begin{remark}
One may contemplate a different (stronger) definition of symmetry. For example, in the first part, we could have required that $\sigma\big(f_1(m,w,u)\big) = f_1\big(\sigma_M(m),w,u\big)$, for all $u\in[0,1]$. While this would lead to a simpler proof, this stronger definition would be too restrictive, as explained in Appendix \ref{app:restr}.
\end{remark}

\begin{remark}
  Note that a symmetry assumption is imposed on the memory update function $f_3$ at the time that a job is dispatched. However, we do not need to impose a similar assumption on the memory update functions $g_2$ and $h_2$ at the times that the dispatcher receives a message. Similarly, there is no symmetry assumption  on the functions $g_1$ and $h_1$ that govern the generation of server messages. In particular, we allow each server to send spontaneous messages at its own identity-dependent, and hence asymmetric, rate.
\end{remark}

\subsection{Delay lower bound for resource constrained policies} \label{sec:imposs}
Before stating the main result of this paper, we introduce formal definitions for the average message rate between the dispatcher and the servers, and for our performance metric for the delay. Furthermore, we introduce an assumption on the arrival process.

First, given a policy of the form specified in the previous subsection, we define the {\bf average message rate} between the dispatcher and the servers as
\begin{align}
 \limsup_{t\to\infty} \frac{1}{t}&\left[ \sum\limits_{k=1}^{A_n(t)} 2|{\bf S}_k| + \sum\limits_{k=1}^{R_n(t)} \mathds{1}_{\mathcal{N}_n}\Big(g_1\big({\bf Q}\big({T^s_k}\big),X_k\big)\Big) \right. \nonumber \\
  &\qquad\qquad\qquad\qquad\quad \left. + \sum\limits_{i=1}^n \sum\limits_{k:\, T^d_k(i)<t}
   \mathds{1}_{\{1\}}\Big(h_1\big({\bf Q}_i\big({T^d_k(i)}\big),Y_k\big)\Big)\right]. \label{eq:messageRate}
\end{align}

Second, we provide a formal definition of our performance metric for the delay. We assume that the process $({\bf Q}(t),M(t),Z(t))_{t\geq 0}$ is stationary, with invariant probability measure $\pi$. Since the destinations of jobs (and their queueing delays) are deterministic functions of the state and i.i.d. randomization variables, the point process of arrivals with the queueing delays as marks, is also stationary. Using this, we define the {\bf expected queueing delay in steady-state $\pi$ of a typical job}, denoted by $\mathbb{E}_{\pi}^0\left[L_0\right]$, as follows. If $L_k$ is the queueing delay of the $k$-th job under the stationary process $({\bf Q}(t),M(t),Z(t))_{t\geq 0}$, then
\begin{equation}\label{eq:delayDefinition}
 \mathbb{E}_{\pi}^0\left[L_0\right]\triangleq \mathbb{E}_{\pi}\left[ \frac{1}{\lambda n t} \sum\limits_{k=1}^{A_n(t)} L_k \right],
\end{equation}
where the right-hand side is independent from $t$ due to the stationarity of the processes involved (see \cite{PASTA}). Furthermore, if the stationary process $({\bf Q}(t),M(t),Z(t))_{t\geq 0}$ is ergodic (in the sense that every invariant set has measure either $0$ or $1$ under $\pi$), we have
\[ \mathbb{E}_{\pi}^0\left[L_0\right] = \lim_{t\to\infty} \frac{1}{A_n(t)} \sum\limits_{k=1}^{A_n(t)} L_k, \qquad a.s. \]

Finally, we introduce an assumption on the arrival process.

\begin{assumption}\label{ass:arrivals}
  Let $I_n$ be distributed as the typical inter-arrival times of the delayed renewal process $A_n(t)$. We assume that there exists a constant $\overline\epsilon>0$, independent from $n$, such that the following holds. For every $\epsilon\in(0,\overline\epsilon]$, there exists a positive constant $\delta_\epsilon$ such that
  \[ \delta_\epsilon \leq \mathbb{P}\left( I_n \leq \frac{\epsilon}{n} \right)\leq 1- \delta_\epsilon, \]
  for all $n$.
\end{assumption}

This assumption implies that arbitrarily small inter-arrival times of order $\Theta(1/n)$ occur with a probability that is bounded away from $0$, and from $1$, for all $n$. In particular, this excludes deterministic inter-arrival times, and inter-arrival times that can take values of order $o(1/n)$ with probability of order $1-o(1)$. On the other hand, if $A(t)$ is a delayed renewal process, where the typical inter-arrival times are continuous random variables with positive density around $0$, then the process $A_n(t)\triangleq A(nt)$ satisfies Assumption \ref{ass:arrivals}.\\

We are now ready to state the main result. It asserts that within the class of symmetric policies that we consider, and under some upper bounds on the memory size (logarithmic) and the message rate (linear), the expected queueing delay in steady-state of a typical job is bounded below by a positive constant.

\begin{theorem}[Positive delay for resource constrained policies]
\label{thm:impossibility}
For any constants $\lambda\in(0,1)$, $c$, $\alpha >0$, and for every arrival process that satisfies Assumption \ref{ass:arrivals}, there exists a constant $\zeta(\lambda,c,\alpha)>0$ with the following property. For any fixed $n$, consider a symmetric memory-based dispatching policy, i.e., that satisfies Assumption \ref{def:symmetry}, with at most $c\log_2(n)$ bits of memory, with an average message rate (cf. Equation \eqref{eq:messageRate}) upper bounded by $\alpha n$ in expectation, and under which the process $({\bf Q}(t),M(t),Z(t))_{t\geq 0}$ admits at least one invariant probability measure $\pi_n$. Then, for all $n$ large enough, we have
\[ \mathbb{E}_{\pi_n}^0\left[L_0\right] \geq \zeta(\lambda,c,\alpha), \]
where $\mathbb{E}_{\pi_n}^0\left[L_0\right]$ is the expected queueing delay in steady-state $\pi_n$ of a typical job.
\end{theorem}
The proof is given in Section \ref{sec:proof_imposs}.

\section{Dispatching policies in the literature}\label{sec:review}

In this section we put our results in perspective by showing that various dispatching policies considered earlier in the literature are special cases of the class of symmetric dispatching policies described above. Most policies have only been studied for the case of Poisson arrivals and exponential service times, so this review is restricted to that case unless stated otherwise.

\subsection{Open-loop policies}

\subsubsection{Random routing} The simplest policy is to dispatch each arriving job to a random queue, with each queue being equally likely to be selected. In this case, the system behaves as $n$  independent parallel M/M/1 queues. This policy needs no messages or memory, and has a positive queueing delay independent of $n$.

\subsubsection{Round Robin (RR)} When the dispatcher has no access to the workload of incoming jobs and no messages are allowed, it is optimal to dispatch arriving jobs to the queues in a round-robin fashion \cite{roundRobin}. This policy does not require messages but needs $\lceil\log_2(n)\rceil$ bits of memory to store the ID of the next queue to receive a job. In the limit, each queue behaves like a D/M/1 queue (see \cite{roundRobin}). While random routing is a symmetric policy, Round Robin is not. To see this, note that a memory state $i$ must be followed by state $i+1$, and such a transition is not permutation-invariant; in particular, the memory update function $f_3$ does not satisfy the symmetry assumption. Round Robin can be made symmetric by using an additional $n\lceil\log_2(n)\rceil$ bits of memory to specify the order with which the different servers are selected. But in any case, this policy also has a positive queueing delay, that does not vanish as $n$ increases.

\subsection{Policies based on queue lengths}

\subsubsection{Join a shortest queue (SQ)} If we wish to minimize the queueing delay and have access to the queue lengths but not to the job sizes, an optimal policy is to have each incoming job join a shortest queue, breaking any ties uniformly at random  \cite{winston}. When $n$ goes to infinity, the queueing delay vanishes, but this policy requires a message rate of $2\lambda n^2$ ($n$ queries and $n$ responses for each arrival), and no memory. This policy is symmetric and achieves vanishing delay, but uses a superlinear number of messages.

\subsubsection{Join a shortest of $d$ random queues (SQ($d$))}
In order to sharply decrease the number of messages sent, Mitzenmacher \cite{mitzenmacher} and Vvedenskaya et al. \cite{vvedenskaya} introduced the power-of-$d$-choices policy. When there is an arrival, $d$ servers are chosen uniformly at random, and the job is sent to a shortest queue among those $d$ servers. This policy fits our framework, and in particular is symmetric; it uses $2\lambda d n$ messages per unit of time, and zero memory. This policy was also analyzed in the case of heavy-tailed service times by Bramson et al. \cite{bramson13}, yielding similar results. In any case, this policy has positive delay, which is consistent with Theorem \ref{thm:impossibility}.

\subsubsection{Join a shortest of $d_n$ random queues (SQ($d_n$))}
More recently, Mukherjee et al. \cite{BorstPowerOfd} analyzed a variation of the SQ($d$) policy, which lets $d$ be a function of the system size $n$. This policy is symmetric, uses $2\lambda d_n n$ messages per unit of time and zero memory, and has zero delay as long as $d_n\to\infty$, which is consistent with Theorem \ref{thm:impossibility}.

\subsubsection{Join a shortest of $d$ queues, with memory (SQ($d$,$b$))}
Another improvement over the power-of-$d$-choices, proposed by Mitzenmacher et al. in \cite{shahFiniteLevels}, is obtained by using extra memory to store the IDs of the $b$ (with $b\leq d$) least loaded queues known at the time of the previous arrival. When a new job arrives, $d$ queues are sampled uniformly at random and the job is sent to a least loaded queue among the $d$ sampled and the $b$ stored queues. This policy is symmetric, needs $2\lambda d n$ messages per unit of time and $\Omega$($b\log_2 n$) bits of memory, and has positive delay, consistent with Theorem \ref{thm:impossibility}.

\subsubsection{SQ($d$) for divisible jobs}
Recently, Ying et al. \cite{srikant15} considered the case of jobs of size $m_n$ (with $m_n\in\omega(1)$ and $m_n/n\to 0$) arriving as a Poisson process of rate $n\lambda/m_n$, where each job can be divided into $m_n$ tasks with mean size $1$. Then, the dispatcher samples $dm_n$ queues and does a water-filling of those queues with the $m_n$ tasks. In this case, the number of messages sent per unit of time is $2\lambda d n$ and no memory is used.
Even though this was not mentioned in \cite{srikant15}, this policy can be shown to drive the queueing delay to $0$ if $d\geq 1/(1-\lambda)$. However, this model does not fall into our framework because it involves divisible jobs.

\subsection{Policies based on remaining workload}

\subsubsection{Join a least loaded queue (LL)}
An appealing policy is the one that sends incoming jobs to a queue with the least remaining workload, in which case the whole system behaves as an M/M/$n$ queue. This policy is symmetric and achieves a vanishing delay as $n\to\infty$, but it has the same quadratic messaging requirements as SQ.

\subsubsection{Join a least loaded of $d$ queues (LL($d$))}
A counterpart of SQ($d$) is LL($d$), in which the dispatcher upon arrival chooses $d$ queues uniformly at random and sends the job to one of those queues with the least remaining workload, breaking any ties uniformly at random. This setting was studied in \cite{LL}, and it does not result in asymptotically vanishing delay, consistent with Theorem \ref{thm:impossibility}.

\subsection{Policies based on job size}
The previous policies dispatched the incoming jobs based on information about the state of the queues, obtained by dynamically exchanging messages with the servers. Such information could include the remaining workload at the different queues. On the other hand, if the dispatcher only knows the size of an incoming job (which might be difficult in practice \cite{predictions}), it could use a static and memoryless policy that selects a target server based solely on the job size. Harchol-Balter et al. \cite{harchol-balter} showed that delay is minimized over all such static policies by a non-symmetric policy that partitions the set of possible job sizes into consecutive intervals and assigns each interval to a different server. This is especially effective when the jobs have highly variable sizes (e.g., heavy-tailed), yet the resulting delay can be no better than that of an M/D/1 queue, and does not vanish as $n\to\infty$. This scheme does not require any message exchanges, and could be made symmetric by using the memory to store a list of the $n$ intervals of job sizes corresponding to each of the $n$ servers.

\subsection{Pull-based load balancing}
In order to reduce the message rate, Badonnel and Burgess \cite{badonnelBurgess}, Lu et al. \cite{joinIdleQueue}, and Stolyar \cite{stolyar14} propose a scheme where messages are sent from a server to the dispatcher whenever the server becomes idle, so that the dispatcher can keep track of the set of idle servers in real time. Then, an arriving job is to be sent to an empty queue (if there is one) or to a queue chosen uniformly at random (if all queues are non-empty). This policy requires at most $\lambda n$ messages per unit of time and exactly $n$ bits of memory (one bit for each queue, indicating whether it is empty or not). Stolyar \cite{stolyar14} has shown that when $n$ goes to infinity, the average delay vanishes. This policy is symmetric. It has a vanishing delay and a linear message rate, but uses superlogarithmic memory, consistent with Theorem \ref{thm:impossibility}.

\subsection{Memory, messages, and queueing delay}

We now summarize the resource requirements (memory and message rate) and the asymptotic delay of the policies reviewed in this section that fall within our framework.

\begin{center}
    \begin{tabular}{ | l | c | c | c |}
    \hline
    Policy & Memory (bits) & Message rate & Limiting delay \\ \hline\hline
    Random & $0$ & $0$ & $>0$ \\ \hline
    RR \cite{roundRobin} & $\log_2 n$ & $0$ & $>0$ \\ \hline
    SQ \cite{winston}& $0$ & $2 \lambda n^2$ & $0$ \\ \hline
    SQ($d$) \cite{mitzenmacher} & $0$ & $2d\lambda n$ & $>0$ \\ \hline
    SQ($d_n$) \cite{BorstPowerOfd} & $0$ & $\omega(n)$ & $0$ \\ \hline
    SQ($d,b$) \cite{shahFiniteLevels} & $b\log_2(n)$ & $2d\lambda n$ & $>0$ \\ \hline
    LL & $0$ & $2 \lambda n^2$ & $0$ \\ \hline
    LL($d$) \cite{LL} & $0$ & $2d\lambda n$ & $>0$ \\ \hline
    Pull-based \cite{stolyar14} & $n$ & $\lambda n$ & $0$ \\ \hline
    \end{tabular}
\end{center}

Note that any one of the above listed policies that achieves vanishing queueing delay falls into one (or both) of the following two categories:
\begin{itemize}
  \item [a)] Those requiring $\omega(n)$ message rate, namely, SQ, SQ($d_n$), and LL.
  \item [b)] Those requiring $\omega(\log_2 n)$ bits of memory (Pull-based).
\end{itemize}
Our main result effectively establishes that this reflects a fundamental limitation of symmetric policies.

\section{Proof of main result}\label{sec:proof_imposs}
Let us fix some $n$. In the sequel, we will assume that $n$ is large enough whenever needed for certain inequalities to hold. We fix a memory-based policy that satisfies Assumption \ref{def:symmetry} (symmetry), with at most $n^c$ memory states, and which results in the process  $({\bf Q}(t),M(t),Z(t))_{t\geq 0}$ having at least one invariant probability measure. Let us fix such an invariant probability measure $\pi_n$. We consider the process in steady-state; that is, we assume that $({\bf Q}(0),M(0),Z(0))$ is distributed according to $\pi_n$. Accordingly, probabilities $\mathbb{P}(\cdot)$ and expectations $\mathbb{E}[\,\cdot\,]$ encountered in the sequel will always refer to the process in steady-state.

The high-level outline of the proof is as follows. In Section \ref{sec:limitation} we show that under our symmetry assumption, the dispatcher can give special treatment to at most $c$ servers, which we call \emph{distinguished servers}. The treatment of all other servers, is symmetric, in some appropriate sense.

In Section \ref{subsec:badEvent} we consider a sequence of bad events under which, over a certain time interval, there are $c+1$ consecutive arrivals,  no service completions or messages from the servers, and all sampled servers are ``busy'' with a substantial workload. Then, in Section \ref{subsec:lowerBound}, we show that this sequence of bad events has non-negligible probability.

In Section \ref{sec:boundUsefulServers}, we develop some further consequences of the symmetry assumption, which we use to constrain the information available to the dispatcher at the time of the $(c+1)$st arrival. Loosely speaking, the idea is that during the interval of interest, the server only has information on $c$ distinguished servers together with (useless) information on some busy servers. This in turn implies (Section \ref{s:compl}) that at least one of the first $c+1$ arrivals must be dispatched to a server on which no useful information is available, and which therefore has a non-negligible probability of inducing a non-negligible delay, thus completing the proof.

\subsection{Local limitations of finite memory} \label{sec:limitation}
We consider the (typical) case where a relatively small number of servers ($\sqrt{n}$ or less) are sampled. We will use the symmetry assumption to show that except for a small set of distinguished servers, of size at most $c$, all other servers must be treated as indistinguishable.

\begin{proposition}\label{prop:symSampling}
Let $U$ be a uniform random variable over $[0,1]$. For all $n$ large enough, for every memory state $m\in\mathcal{M}_n$ and  every possible job size $w \in\mathbb{R}_+$, the following holds. Consider any vector of servers ${\bf s}\in\mathcal{S}_n$ (and its associated set of servers ${\bf s}^{set}$) with $|{\bf s}|\leq \sqrt{n}$, and any integer $\ell$ with $|{\bf s}|+1\leq \ell\leq n$.
Consider the event $B(m,w;{\bf s},\ell)$ that exactly $\ell$ servers are sampled and that the first $|{\bf s}|$ of them  are the same as the vector ${\bf s}$, i.e.,
\[ B(m,w;{\bf s},\ell) = \big\{ |f_1(m,w,U)|=\ell \big\} \cap \bigcap_{i=1}^{|{\bf s}|}  \big\{ f_1(m,w,U)_i={\bf s}_i \big\}, \]
and assume that the conditional probability measure
\[ \mathbb{P}\left(\,\, \cdot \,\, \big| \ B(m,w;{\bf s},\ell)\right) \]
is well-defined. Then, there exists a unique set $R(m,w,{\bf s},\ell)\subset \mathcal{N}_n \backslash {\bf s}^{set}$ of minimal cardinality such that
\begin{equation}\label{eq:ss}
 \mathbb{P}\left(f_1(m,w,U)_{|{\bf s}|+1} = j \ \big| \ B(m,w;{\bf s},\ell)\right)
 \end{equation}
 is the same for all $j \notin R(m,w,{\bf s},\ell)\cup {\bf s}^{set}$. Furthermore, $|R(m,w,{\bf s},\ell)|\leq c$.
\end{proposition}

\begin{remark}\label{rem}
With some notational abuse, the measure $\mathbb{P}$ in Proposition \ref{prop:symSampling} need not correspond to the measure $\mathbb{P}$ that describes the process. We are simply considering probabilities associated with a deterministic function of the uniform random variable $U$.
\end{remark}

\begin{proof}
Throughout the proof, we fix a particular memory state $m$, job size $w$, {vector} of servers ${\bf s}$ with $|{\bf s}|\leq\sqrt{n}$, and an integer $\ell$ in the range $|{\bf s}|+1\leq \ell\leq n$. To simplify notation, we will suppress the dependence on $w$.

Consider the random vector ${\bf S}(m)\triangleq f_1(m,U)$. Let $\bf v$ be the vector whose components are indexed by $j$ ranging in the set $\big({\bf s}^{set}\big)^c=\mathcal{N}_n\backslash {\bf s}^{set}$, and defined for any such $j$, by
\[ {\bf v}_j= \mathbb{P}\left( {{\bf S}(m)_{|{\bf s}|+1} = j} \ \big| \
B(m; {\bf s},\ell)\right).\]
We need to show that for $j$ outside a ``small'' set, all of the components of ${\bf v}$ are equal.
Let $z_1,\ldots,z_d$ be the distinct values of ${\bf v}_j$, as $j$ ranges over $\big({\bf s}^{set}\big)^c$, and let $A_\alpha =\{j \in \big({\bf s}^{set}\big)^c \mid {\bf v}_j=z_\alpha\}$. The sequence of sets $(A_1,\ldots,A_d)$ provides a partition of $\big({\bf s}^{set}\big)^c$ into equivalence classes, with ${\bf v}_j={\bf v}_{j'}=z_{\alpha}$, for all $j,j'$ in the $\alpha$-th equivalence class $A_\alpha$. Let $k_1,\ldots,k_d$ be the cardinalities of the equivalence classes $A_1,\ldots,A_d$. Without the loss of generality, assume that $k_d$ is a largest such cardinality. We define
\[ R=\Big\{j\in \big({\bf s}^{set}\big)^c \mid {\bf v}_j \neq {\bf v}_d\Big\} = A_1\cup\cdots\cup A_{d-1}, \]
so that $R^c\cap \big({\bf s}^{set}\big)^c=A_d$.
For every $j,j'\in A_d$, we have  ${\bf v}_j={\bf v}_{j'}={\bf v}_d$, and therefore the condition \eqref{eq:ss} is satisfied by $R$. Note that by choosing ${\bf v}_d$ to be the most common value, we are making the cardinality of the set $R^c\cap \big({\bf s}^{set}\big)^c=\{j\notin {\bf s}^{set}\mid {\bf v}_j={\bf v}_d\}$ as large as possible, from which it follows that the set $R\cap \big({\bf s}^{set}\big)^c$ is as small as possible, and therefore $R$, as defined, is indeed a minimal cardinality subset of $\big({\bf s}^{set}\big)^c$ that satisfies \eqref{eq:ss}.

We now establish the desired upper bound on the cardinality of $R$. Let $\Sigma$ be the set of permutations that fix the set ${\bf s}^{set}$. Consider an arbitrary permutation $\sigma\in\Sigma$ and let $\sigma_M$ be a corresponding permutation of the memory states, as defined by Assumption \ref{def:symmetry}. We let ${\bf v}_{\sigma^{-1}}$ be the vector with components $({\bf v}_{\sigma^{-1}})_j={\bf v}_{\sigma^{-1}(j)}$, for $j \notin {\bf s}^{set}$. Note that as we vary $\sigma$ over the set $\Sigma$, ${\bf v}_{\sigma^{-1}}$ ranges over  all possible permutations of the vector ${\bf v}$. We also have, for $j \notin {\bf s}^{set}$,
\begin{align*}
({\bf v}_{\sigma^{-1}})_j & ={\bf v}_{\sigma^{-1}(j)} \\
&=  \mathbb{P}\Big( {\bf S}(m)_{|{\bf s}|+1} = \sigma^{-1}(j) \ \Big| \ B(m; {\bf s}, \ell)  \Big) \\
&= \mathbb{P}\left( {\bf S}(m)_{|{\bf s}|+1} = \sigma^{-1}(j) \ \left| \ \big\{ |{\bf S}(m)|=\ell \big\} \cap \bigcap_{i=1}^{|{\bf s}|} \big\{ {\bf S}(m)_i={\bf s}_i \big\} \right. \right) \\
&= \mathbb{P}\left( \sigma\big({\bf S}(m)_{|{\bf s}|+1}\big) = j \ \left| \ \big\{ |{\bf S}(m)|=\ell \big\} \cap \bigcap_{i=1}^{|{\bf s}|} \big\{ {\bf S}(m)_i={\bf s}_i \big\} \right. \right)
\end{align*}
\begin{align*}
&= \mathbb{P}\left( \sigma\big({\bf S}(m)_{|{\bf s}|+1}\big) = j \ \left| \ \big\{ |\sigma({\bf S}(m))|=\ell \big\} \cap \bigcap_{i=1}^{|{\bf s}|} \big\{ \sigma({\bf S}(m)_i)=\sigma({\bf s}_i) \big\} \right. \right)\\
&= \mathbb{P}\left( \sigma\big({\bf S}(m)_{|{\bf s}|+1}\big) = j \ \left| \ \big\{ |\sigma({\bf S}(m))|=\ell \big\} \cap \bigcap_{i=1}^{|{\bf s}|} \big\{ \sigma({\bf S}(m)_i)={\bf s}_i \big\} \right. \right) \\
&= \mathbb{P}\left( {\bf S}\big(\sigma_M(m)\big)_{|{\bf s}|+1} = j \ \left| \ \big\{ |{\bf S}(\sigma_M(m))|=\ell \big\} \cap \bigcap_{i=1}^{|{\bf s}|} \big\{ {\bf S}(\sigma_M(m))_i={\bf s}_i \big\} \right. \right).
\end{align*}
Note that in the above expressions, the only random variables are ${\bf S}(m)$ and ${\bf S}\big(\sigma_M(m)\big)$, while ${\bf s}$ is a fixed vector. The next to last equality above holds because $\sigma$ fixes the elements in the vector $\bf s$; the last equality follows because the random variables $\sigma\big({\bf S}(m)\big)$ and ${\bf S}\big(\sigma_M(m)\big)$ are identically distributed, according to Part 1 of the symmetry Assumption \ref{def:symmetry}. The equality that was established above implies that $\sigma_M(m)$ completely determines the vector ${\bf v}_{\sigma^{-1}}$. As $\sigma\in\Sigma$ changes, $\sigma_M(m)$ can take at most $n^c$ distinct values, due to the assumed bound on the memory size, and this leads to a bound on the number of possible permutations of the vector $\bf v$:
\begin{equation}\label{eq:extra1}
  \big|\{{\bf v}_{\sigma^{-1}}:\sigma \in\Sigma\}\big|\leq n^c. \nonumber
\end{equation}

We now argue that since $\bf v$ has relatively few distinct permutations, most of its entries ${\bf v}_i$ must be equal. Recall the
partition of the set $\big({\bf s}^{set}\big)^c$ of indices into equivalence classes, of sizes $k_1,\ldots,k_d$, with $k_d$ being the largest cardinality. Note that there is a one-to-one correspondence between distinct permutations ${\bf v}_{\sigma^{-1}}$ of the vector $\bf v$ and distinct partitions of $\big({\bf s}^{set}\big)^c$ into a sequence of subsets of cardinalities $k_1,\ldots,k_d$, with the value $z_\alpha$ being taken on the $\alpha$-th subset. It follows that the number of different partitions of $S^c$ into sets with the given cardinalities, which is given by the multinomial coefficient, satisfies
$${n-|{\bf s}| \choose k_1!\ k_2! \ \cdots k_d!} = \big|\{{\bf v}_{\sigma^{-1}}:\sigma \in\Sigma\}\big|\leq n^c. $$
The number of choices of a $k_d$-element subset is no larger than the number of partitions. Therefore,
$${n-|{\bf s}| \choose k_d} \leq n^c.$$
An elementary calculation (cf.~Lemma \ref{lem:chooseBound}) implies that when $n$ is large enough, we must have either (i) $k_d\geq n-|{\bf s}|-c$ or (ii) $k_d \leq c$. We argue that the second possibility cannot occur. Indeed, if $k_d\leq c$, and since $k_d$ is the largest cardinality, it follows that $k_\alpha\leq c$ for every $\alpha$. Since $k_1+\cdots+k_d =n-|{\bf s}|$, we obtain that the number of classes, $d$, is at least $\lceil (n-|{\bf s}| )/c \rceil$. When dealing with $d$ different classes, the number of possible partitions is at least $d!$; this can be seen by focusing on the least-indexed entry in each of the $d$ classes and noting that these $d$ entries may appear in an arbitrary order.
Since $|{\bf s}|\leq \sqrt{n}$, we have $n-|{\bf s}|\geq n/2$, and putting everything together, we obtain
$$ \lceil (n/2c) \rceil ! \leq \left\lceil \frac{n-|{\bf s}|}{c} \right\rceil !
 \leq n^c.$$
This is clearly impossible when $n$ is large enough, and case (ii) can therefore be eliminated. We conclude that $|A_d|=k_d\geq n-|{\bf s}|-c$. Since $|A_1\cup\cdots \cup A_{d}|=|\big({\bf s}^{set}\big)^c|=n-|{\bf s}|$,
it follows that $|R|=|A_1\cup\cdots\cup A_{d-1}|\leq c$, which is the desired cardinality bound on $R$.

It should be apparent that any minimal cardinality set $R$ that satisfies \eqref{eq:ss} must be constructed exactly as our set $A_d$. Thus, non-uniqueness of the set $R$ with the desired properties will arise if and only if there is another subset  $A_\alpha$, with $\alpha\neq d$, with the same maximal cardinality $k_d$. On the other hand, since $|{\bf s}|\leq \sqrt{n}$, we have $k_d \geq n-|{\bf s}|-c >n/2$, when $n$ is large enough. But having two disjoint subsets, $A_{d}$ and $A_{\alpha}$, each of cardinality larger than $n/2$ is impossible, which proves uniqueness.
\end{proof}

Using a similar argument, we can also show that the distribution of the destination of the incoming job is uniform (or zero) outside the set of sampled servers and a set of at most $c$ distinguished servers.

\begin{proposition}\label{prop:symDispatching}
Let $V$ be a uniform random variable over $[0,1]$. For all $n$ large enough, for every memory state $m\in\mathcal{M}_n$, every vector of indices ${\bf s}\in\mathcal{S}_n$ with $|{\bf s}|\leq \sqrt{n}$, every queue vector state ${\bf q}\in\mathcal{Q}^{|{\bf s}|}$, and  every job size $w \in\mathbb{R}_+$, the following holds. There exists a unique set $R'\big(m,w,{\bf s},{\bf q}\big)\subset\mathcal{N}_n\backslash {\bf s}^{set}$ of minimal cardinality such that
\begin{align*}
 &\mathbb{P}\Big(f_2\big(m,w,{\bf s},{\bf q},V\big)=j\Big) = \mathbb{P}\Big(f_2\big(m,w,{\bf s},{\bf q},V\big)=k\Big),
\end{align*}
for all $j,k\notin R'\big(m,w,{\bf s},{\bf q}\big)\cup {\bf s}^{set}$. Furthermore, $|R'(m,w,{\bf s},{\bf q})|\leq c$.
\end{proposition}
\begin{proof}
  The proof is analogous to the proof of the previous proposition. We start by defining a vector $\bf v$, whose components are again indexed by $j$ ranging in the set $\mathcal{N}_n\backslash {\bf s}^{set}$, by
  \[ {\bf v}_j = \mathbb{P}\Big(f_2\big(m,w,{\bf s},{\bf q},V\big)=j\Big). \]
  Other than this new definition of the vector $\bf v$, the rest of the proof follows verbatim the one for Proposition \ref{prop:symSampling}.
\end{proof}

\subsection{A sequence of ``bad" events}\label{subsec:badEvent}
In this subsection we introduce a sequence of ``bad" events that we will be focusing on in order to establish a positive lower bound on the delay.

Recall that $T^s_1$ is the time of the first event of the underlying Poisson process of rate $\mu n$ that generates the spontaneous messages from the servers. Recall also that we denote by ${\bf Q}_{i,1}(t)$ the remaining workload of the job being serviced in server $i$, at time $t$, with ${\bf Q}_{i,1}(t)=0$ if no job is present at server $i$. Let
\begin{equation}\label{eq:B}
 B\triangleq \left\{ i: \sum_{j=1}^\infty {\bf Q}_{i,j}(0) \geq 2\gamma \right\},
\end{equation}
which is the set of servers with at least $2\gamma$ remaining workload in their queues, and let $N_b=|B|$, where $\gamma\leq 1$ is a small positive constant, independent of $n$, to be specified later.

Consider the following events:
\begin{itemize}
  \item [(i)] the first $c+1$ jobs after time $0$ are all of size at least $2\gamma$,
  \[ \mathcal{A}_w \triangleq \{W_1,\dots,W_{c+1}\geq 2\gamma\}; \]
  \item [(ii)] the first potential spontaneous message occurs after time $\gamma/n$, and the $(c+1)$-st arrival occurs before time $\gamma/n$,
      \[ \mathcal{A}_a \triangleq \left\{T^s_1>\frac{\gamma}{n} \right\} \cap \left\{ T_{c+1}< \frac{\gamma}{n}\right\}; \]
  \item [(iii)] there are no service completions before time $\gamma/n$,
  \[ \mathcal{A}_s \triangleq \left\{{\bf Q}_{i,1}(0)\notin \Big(0,\frac{\gamma}{n}\Big), \,\forall \, i \right\}; \]
  \item [(iv)] there are at least $\gamma n$ servers that each have at least $2\gamma$ remaining workload at time zero,
      \[ \mathcal{A}_b \triangleq \Big\{N_b \geq \gamma n\Big\}. \]
\end{itemize}
For an interpretation, the event
\[ \mathcal{H}^+_0 \triangleq \mathcal{A}_w\cap\mathcal{A}_a\cap\mathcal{A}_s \cap\mathcal{A}_b, \]
corresponds to an unfavorable situation for the dispatcher. This is because, at time zero, the dispatcher's memory contains possibly useful information on at most $c$ distinguished servers (Propositions \ref{prop:symSampling} and \ref{prop:symDispatching}), and has to accommodate $c+1$ arriving jobs by time $\gamma/n$. On the other hand, a nontrivial fraction of the servers are busy and will remain so until time $\gamma/n$ (event $\mathcal{A}_b$), and it is possible that sampling will not reveal any idle servers (as long as the number of sampled servers is not too large). Thus, at least one of the jobs may end up at a busy server, resulting in positive expected delay. In what follows, we go through the just outlined sequence of unfavorable events, and then, in Subsection \ref{subsec:lowerBound}, we lower bound its probability.

Starting with $\mathcal{H}_0^+$, we define a nested sequence of events, after first introducing some more notation. For $k=1,\dots,c+1$, let ${\bf S}_k$ be the random (hence denoted by an upper case symbol) vector of servers that are sampled upon the arrival of the $k$-th job; its components are denoted by $({\bf S}_k)_i$. For $i=0,1,\ldots, |{\bf S}_k|$, we let
\[ R_{k,i} \triangleq R\Big(M\big(T_k^-\big),W_k,\big(({\bf S}_k)_1,\dots,({\bf S}_k)_{i-1}\big),|{\bf S}_k|\Big) \]
be the (random) subset of servers defined in Proposition \ref{prop:symSampling} (whenever $M\big(T_k^-\big)$, $W_k$, $\big(({\bf S}_k)_1,\dots,({\bf S}_k)_{i-1}\big)$, and $|{\bf S}_k|$ are such that the proposition applies), with the convention that $\big(({\bf S}_k)_1,\dots,({\bf S}_k)_{i-1}\big)=\emptyset$ when $i=1$. Otherwise, we let $R_{k,i}\triangleq \emptyset$. Furthermore, we define
\[ R_k \triangleq \bigcup_{i=1}^{|{\bf S}_k|} R_{k,i}. \]
Moreover, let $D_k$ be the destination of the $k$-th job, and let
\[ R'_k \triangleq R'\Big(M\big(T_k^-\big),W_k,{\bf S}_k,{\bf Q}_{{\bf S}_k}(T_k^-)\Big) \]
be the (random) subset of servers defined in Proposition \ref{prop:symDispatching} (whenever $M\big(T_k^-\big)$, $W_k$, ${\bf S}_k$ and ${\bf Q}_{{\bf S}_k}(T_k^-)$ are such that the proposition applies). Otherwise, we let $R'_k\triangleq \emptyset$. Finally, given a collection of constants $\xi_1,\ldots,\xi_{c+1}$, independent of $n$ and to be determined later, we define a nested sequence of events recursively, by
\begin{align}
 \mathcal{H}_k^- &\triangleq \mathcal{H}_{k-1}^+ \cap \big\{|{\bf S}_k|\leq \xi_k\big\}, \nonumber\\
 \mathcal{H}_k  & \triangleq \mathcal{H}_k^- \cap \big\{ ({\bf S}_k)_i \in R_{k,i}\cup B, \,\, \, i=1,\dots,|{\bf S}_k| \big\},  \nonumber\\
 \mathcal{H}_k^+ &\triangleq \mathcal{H}_k \cap \big\{D_k\in {\bf S}_k^{set}\cup R'_k\cup B\big\}, \label{eq:hkplus}
\end{align}
for $k=1,\dots,c+1$.

\subsection{Lower bound on the probability of ``bad" events}\label{subsec:lowerBound}
In this subsection, we establish a positive lower bound, valid for all $n$ large enough, for the probability of the event $\mathcal{H}^+_{c+1}$. In order to do this, we will obtain such uniform lower bounds for the probability of $\mathcal{H}^+_{k}$, for $k\geq 0$, by induction. We start with the base case.

\begin{lemma}\label{lem:base_case}
  There exists a constant $\alpha_0^+>0$, independent of $n$, such that
  \[ \mathbb{P}(\mathcal{H}_0^+)\geq \alpha_0^+. \]
\end{lemma}
\begin{proof}
  Note that the event $\mathcal{A}_a$ only depends on the processes of arrivals and spontaneous messages after time zero, $\mathcal{A}_w$ only depends on the i.i.d. workloads $W_1,\dots,W_{c+1}$, and $\mathcal{A}_s\cap\mathcal{A}_b$ only depends on the initial queue length vector ${\bf Q}(0)$. It follows that
\[ \mathbb{P}(\mathcal{H}_0^+) = \mathbb{P}(\mathcal{A}_a) \mathbb{P}(\mathcal{A}_w) \mathbb{P}(\mathcal{A}_s\cap\mathcal{A}_b). \]
We will now lower bound each of these probabilities.

Note that $\mathbb{P}(\mathcal{A}_a)$ is the intersection of two independent events. The first is the event that the first arrival in a Poisson process with rate $\mu n$ happens after time $\gamma/n$, or equivalently, it is the event that the first arrival of a Poisson process of rate $\mu$ happens after time $\gamma$, which has positive probability that does not depend on $n$. The second is the event that $c+1$ arrivals of the delayed renewal process $A_n(t)$ occur before time $\gamma/n$, i.e., the event that $T_{c+1}<\gamma/n$. Since the process $({\bf Q}(t),M(t),Z(t))_{t\geq 0}$ is stationary, the first arrival time ($T_1$) is distributed according to the residual time of typical inter-arrival times. In particular, if $F$ is the cdf of typical inter-arrival times of $A_n(t)$ (which have mean $1/\lambda n$), the well-known formula for the distribution of residual times gives
\begin{align*}
 \mathbb{P}\left( T_1<\frac{\gamma}{n(c+1)} \right) &= \lambda n \int\limits_0^{\frac{\gamma}{n(c+1)}} \Big(1-F(u)\Big) du \\
 &= \lambda \int\limits_0^{\frac{\gamma}{c+1}} \left(1-F\left(\frac{v}{n}\right) \right) dv.
\end{align*}
Recall that Assumption \ref{ass:arrivals} states that $1-F(v/n)\geq\delta_v>0$, for all $v>0$ sufficiently small, and for all $n$. As a result, we have
\begin{align}
  \lambda \int\limits_0^{\frac{\gamma}{c+1}} \left(1-F\left(\frac{v}{n}\right)\right) dv &\geq \lambda \int\limits_0^{\frac{\gamma}{c+1}} \left(1-F\left(\frac{\gamma}{n(c+1)}\right)\right) dv \\
  &= \frac{\lambda\gamma}{c+1} \delta_\frac{\gamma}{c+1}, \label{eq:arrivalInequality}
\end{align}
for all $\gamma$ sufficiently small. On the other hand, for $k=2,\dots,c+1$, Assumption \ref{ass:arrivals} also implies that
\[ \mathbb{P}\left( T_k-T_{k-1} \leq \frac{\gamma}{n(c+1)} \right) \geq  \delta_\frac{\gamma}{(c+1)}. \]
Combining this with Equation \eqref{eq:arrivalInequality}, and using the fact that the first arrival time and the subsequent inter-arrival times are independent, we obtain
\begin{align*}
  \mathbb{P}\left( T_{c+1}<\frac{\gamma}{n} \right) &\geq \mathbb{P}\left( \left\{ T_1<\frac{\gamma}{n(c+1)} \right\} \cap \bigcap_{k=2}^{c+1} \left\{ T_k-T_{k-1} \leq \frac{\gamma}{n(c+1)} \right\} \right) \\
  &\geq \mathbb{P}\left( T_1<\frac{\gamma}{n(c+1)} \right) \prod\limits_{k=2}^{c+1} \mathbb{P}\left( T_k-T_{k-1} \leq \frac{\gamma}{n(c+1)} \right) \\
  &\geq \frac{\lambda\gamma}{c+1} \left(\delta_\frac{\gamma}{c+1}\right)^{c+1},
\end{align*}
which is a positive constant independent from $n$.

We also have
\begin{align*}
 \mathbb{P}(\mathcal{A}_w)&=\prod_{i=1}^{c+1} \mathbb{P}(W_i\geq 2\gamma) \\
 & =\mathbb{P}(W_i\geq 2\gamma)^{c+1},
\end{align*}
which is independent of $n$, and positive for $\gamma$ small enough.

We now consider the event $\mathcal{A}_s$. If $\mathcal{A}_s^c$ holds, then there exists a server $i$ such that $0<{\bf Q}_{i,1}(0)\leq \gamma/n$, and thus we have a job departure during $(0,\frac{\gamma}{n}]$. Let $X$ be the number of service completions during $(0,\frac{\gamma}{n}]$. The occurrence of $\mathcal{A}_s^c$ implies $X\geq 1$. Furthermore, the expected number of service completions in steady-state during any fixed interval must be equal to the expected number of arrivals, so that
\begin{equation}\label{eq:intermediate1}
\mathbb{P}(\mathcal{A}_s^c)\leq \mathbb{E}[X]=(n\lambda)\frac{\gamma}{n}=
\lambda \gamma.
\end{equation}

We now consider the event $\mathcal{A}_b$. Recall that
\[ N_b=\left|\left\{i:\sum_{j=1}^\infty {\bf Q}_{i,j}(0)\geq 2\gamma\right\}\right|. \]
Let
\[ N_I=\left|\left\{i:\sum_{j=1}^\infty {\bf Q}_{i,j}(0)=0\right\}\right|, \]
and
\[ N_d=\left|\left\{i:0<\sum_{j=1}^\infty {\bf Q}_{i,j}(0)<2\gamma\right\}\right|. \]
Then, $n=N_b+N_I+N_d$. Furthermore, all servers with $0<\sum_{j=1}^\infty {\bf Q}_{i,j}(0)<2\gamma$ will have a departure in $(0,2\gamma)$. Let $Y$ be the number of departures (service completions) during $(0,2\gamma)$. Then, $Y\geq N_d$. We use once more that the expected number of service completions in steady-state during any fixed interval must be equal to the expected number of arrivals, to obtain
\[ n\lambda 2\gamma = \mathbb{E}[Y] \geq \mathbb{E}[N_d]. \]
Furthermore, by applying Little's law to the number of busy servers, in steady-state, we obtain
\[ \mathbb{E}[N_I]=(1-\lambda) n. \]
Hence
\begin{equation*}
  \mathbb{E}[N_b] = n-\mathbb{E}[N_I]-\mathbb{E}[N_d] \geq n(\lambda-2\lambda \gamma).
\end{equation*}
On the other hand, we have
\begin{align*}
  \mathbb{E}[N_b] &\leq \mathbb{P}(N_b\leq \gamma n) \gamma n + \mathbb{P}(N_b>\gamma n)n  \\
&\leq \gamma n + \mathbb{P}(N_b \geq \gamma n)n \\
&= \gamma n + \mathbb{P}(\mathcal{A}_b)n.
\end{align*}
Combining these last two inequalities, we obtain
\begin{equation}\label{eq:intermediate2}
 \mathbb{P}(\mathcal{A}_b)\geq \lambda-2\lambda \gamma  - \gamma.
\end{equation}
Finally, using Equations \eqref{eq:intermediate1} and \eqref{eq:intermediate2}, we have
\begin{align*}
  \mathbb{P}(\mathcal{A}_s\cap\mathcal{A}_b) &= \mathbb{P}(\mathcal{A}_b) - \mathbb{P}(\mathcal{A}_b \cap \mathcal{A}_s^c) \\
&\geq \mathbb{P}(\mathcal{A}_b) - \mathbb{P}(\mathcal{A}_s^c) \\
&\geq \lambda -2\lambda \gamma - \gamma - \gamma \lambda,
\end{align*}
which is a positive constant if $\gamma$ is chosen small enough.
\end{proof}

We now carry out the inductive step, from $k-1$ to $k$, in a sequence of three lemmas. We make the induction hypothesis that there exists a positive constant $\alpha_{k-1}^+$ such that $\mathbb{P}(\mathcal{H}_{k-1}^+) \geq \alpha_{k-1}^+$, and we sequentially prove that there exist positive constants $\alpha_k^-$, $\alpha_k$, and $\alpha_k^+$ such that $\mathbb{P}(\mathcal{H}_{k}^-) \geq \alpha_k^-$ (Lemma \ref{lem:induction_step1}), $\mathbb{P}(\mathcal{H}_{k}) \geq \alpha_k$ (Proposition \ref{prop:induction_step2}), and $\mathbb{P}(\mathcal{H}_{k}^+) \geq \alpha_k^+$ (Lemma \ref{lem:induction_step3}).

\begin{lemma}\label{lem:induction_step1}
Suppose that $\mathbb{P}(\mathcal{H}_{k-1}^+) \geq \alpha_{k-1}^+>0$ and that the constant $\xi_k$ is chosen to be large enough. Then, there exists a constant $\alpha_k^->0$, such that for all $n$ large enough, we have $\mathbb{P}(\mathcal{H}_{k}^-) \geq \alpha_k^-$.
\end{lemma}
\begin{proof}
First, recall our assumption that the average message rate (cf. Equation \eqref{eq:messageRate}) is upper bounded by $\alpha n$ in expectation. Therefore,
\[ \mathbb{E}\left[ \limsup_{t\to\infty} \frac{1}{t} \sum\limits_{j=1}^{A_n(t)} 2 |{\bf S}_j| \right] \leq \alpha n, \]
where $A_n(t)$ is the number of arrivals until time $t$. By Fatou's lemma, we also have
\begin{equation*}
 \limsup_{t\to\infty} \mathbb{E}\left[ \frac{1}{t} \sum\limits_{j=1}^{A_n(t)} 2 |{\bf S}_j| \right] \leq \alpha n.
\end{equation*}
Recall that the process $({\bf Q}(t),M(t),Z(t))_{t\geq 0}$ is stationary. Then, since the sampled vectors are a deterministic function of the state, and i.i.d. randomization variables, the point process of arrivals with the sampled vectors as marks, is also stationary. As a result, the expression
\begin{equation*}
 \mathbb{E}\left[ \frac{1}{t} \sum\limits_{j=1}^{A_n(t)} 2 |{\bf S}_j| \right]
\end{equation*}
is independent from $t$ (see Equation (1.2.9) of \cite{PASTA}). In particular, for $t=\gamma/n$, we have that
\begin{equation}\label{eq:messageBound}
 \mathbb{E}\left[ \frac{1}{\gamma} \sum\limits_{j=1}^{A_n\left(\frac{\gamma}{n}\right)} 2 |{\bf S}_j| \right] \leq \alpha.
\end{equation}
Moreover, since $k\leq c+1$, we have
\begin{align*}
  \mathbb{E}\left[ \sum\limits_{j=1}^{A_n\left(\frac{\gamma}{n}\right)} |{\bf S}_j| \right] &\geq \mathbb{E}\left[ \left. \sum\limits_{j=1}^{A_n\left(\frac{\gamma}{n}\right)} |{\bf S}_j| \,\right|\, A_n\left(\frac{\gamma}{n}\right)\geq c+1 \right] \mathbb{P}\left( A_n\left(\frac{\gamma}{n}\right)\geq c+1 \right) \\
  &\geq \mathbb{E}\left[ |{\bf S}_k| \,\left|\, A_n\left(\frac{\gamma}{n}\right)\geq c+1 \right. \right] \mathbb{P}\left( A_n\left(\frac{\gamma}{n}\right)\geq c+1 \right).
\end{align*}
Combining this with Equation \eqref{eq:messageBound}, we obtain
\[ \mathbb{E}\left[ \frac{2}{\gamma} |{\bf S}_k| \,\left|\, A_n\left(\frac{\gamma}{n}\right)\geq c+1 \right. \right] \mathbb{P}\left( A_n\left(\frac{\gamma}{n}\right)\geq c+1 \right) \leq \alpha. \]
This yields the upper bound
\begin{equation}\label{eq:samplingUpperBound}
 \mathbb{E}\left[ |{\bf S}_k| \,\left|\, A_n\left(\frac{\gamma}{n}\right)\geq c+1 \right. \right] \leq \frac{\alpha \gamma}{2 \mathbb{P}\left( A_n\left(\frac{\gamma}{n}\right)\geq c+1 \right)}.
\end{equation}

On the other hand, using the fact that $\mathcal{H}_{k-1}^+\subset \{A_n(\gamma/n)\geq c+1\}$, we have
\begin{align}
  \mathbb{P}\big(\mathcal{H}_k^-\big)& = \mathbb{P}\Big(\mathcal{H}_{k-1}^+ \cap \big\{|{\bf S}_k|\leq \xi_k\big\}\Big) \nonumber \\
  &= \mathbb{P}\left(\mathcal{H}_{k-1}^+ \cap \left\{ A_n\left(\frac{\gamma}{n}\right)\geq c+1 \right\} \cap \big\{|{\bf S}_k|\leq \xi_k\big\}\right) \nonumber \\
  &= \mathbb{P}\left( \mathcal{H}_{k-1}^+ \cap \big\{|{\bf S}_k|\leq \xi_k\big\} \,\left|\, A_n\left(\frac{\gamma}{n}\right)\geq c+1 \right. \right) \mathbb{P}\left( A_n\left(\frac{\gamma}{n}\right)\geq c+1 \right) \nonumber \\
&\geq \mathbb{P}\big(\mathcal{H}_{k-1}^+\big) - \mathbb{P}\left( |{\bf S}_k| > \xi_k \,\left|\, A_n\left(\frac{\gamma}{n}\right)\geq c+1 \right. \right) \mathbb{P}\left( A_n\left(\frac{\gamma}{n}\right)\geq c+1 \right). \label{eq:firstInequality}
\end{align}
Furthermore, for any constant $\xi_k>0$, Markov's inequality implies
\begin{align}
  &\mathbb{P}\left( |{\bf S}_k| > \xi_k \,\left|\, A_n\left(\frac{\gamma}{n}\right)\geq c+1 \right. \right) \leq \frac{\mathbb{E}\left[ \, |{\bf S}_k| \,\left|\, A_n\left(\frac{\gamma}{n}\right)\geq c+1 \right. \right] }{\xi_k}, \label{eq:lastInequality}
\end{align}
which combined with Equation \eqref{eq:firstInequality} yields
\begin{equation}
  \mathbb{P}\big(\mathcal{H}_k^-\big) \geq \mathbb{P}\big(\mathcal{H}_{k-1}^+\big) - \frac{\mathbb{E}\left[ \, |{\bf S}_k| \,\left|\, A_n\left(\frac{\gamma}{n}\right)\geq c+1 \right. \right] }{\xi_k} \mathbb{P}\left( A_n\left(\frac{\gamma}{n}\right)\geq c+1 \right).
\end{equation}
Applying the inequality \eqref{eq:samplingUpperBound} to the equation above, we obtain
\begin{equation}
  \mathbb{P}\big(\mathcal{H}_k^-\big) \geq \mathbb{P}\big(\mathcal{H}_{k-1}^+\big) - \frac{\alpha\gamma}{2 \xi_k}.
\end{equation}
Finally, combining this with the fact that $\mathbb{P}(\mathcal{H}_{k-1}^+) \geq \alpha_{k-1}^+>0$, we have that
\begin{align*}
  \mathbb{P}\big(\mathcal{H}_k^-\big)& \geq \alpha_{k-1}^+ - \frac{\alpha\gamma}{2 \xi_k} \triangleq \alpha_k^-,
\end{align*}
which is positive for all $\xi_k$ large enough.
\end{proof}

\begin{proposition}\label{prop:induction_step2}
Suppose that $\mathbb{P}(\mathcal{H}_{k}^-) \geq \alpha_k^-$, and that the constant $\xi_k$ is chosen large enough. Then, there exists a constant $\alpha_k>0$, such that for all $n$ large enough, we have $\mathbb{P}(\mathcal{H}_{k}) \geq \alpha_k$.
\end{proposition}
\begin{proof}
Recall the definitions
\begin{align*}
 \mathcal{H}_k &= \mathcal{H}_k^- \cap \big\{ ({\bf S}_k)_i \in R_{k,i}\cup B, \,\, \, i=1,\dots,|{\bf S}_k| \big\},
\end{align*}
and
\begin{align*}
 \mathcal{H}_k^- &= \mathcal{H}_{k-1}^+ \cap \{|{\bf S}_k|\leq \xi_k\}.
\end{align*}
For $i=1,\dots,|{\bf S}_k|$, let us denote
\[ H_{k,i} \triangleq \big\{ ({\bf S}_k)_i \in R_{k,i}\cup B \big\}. \]
Then,
\begin{align}
  &\mathbb{P}(\mathcal{H}_{k}) = \mathbb{P}\Big(\mathcal{H}_{k}^- \cap \big\{ ({\bf S}_k)_i \in R_{k,i}\cup B, \,\, i=1,\dots,|{\bf S}_k| \big\} \Big) \nonumber \\
  &\quad=\sum\limits_{\ell} \mathbb{P}\left(\mathcal{H}_{k-1}^+ \cap \{ |{\bf S}_k|=\ell \} \cap \bigcap_{i=1}^{\ell} H_{k,i} \right) \nonumber \\
  &\quad= \sum\limits_{\ell} \mathbb{P}\left( \left. \bigcap_{i=1}^{\ell} H_{k,i} \,\, \right| \,\, \mathcal{H}_{k-1}^+ \cap \{ |{\bf S}_k|=\ell \} \right) \mathbb{P}\big( \mathcal{H}_{k-1}^+ \cap \{ |{\bf S}_k|=\ell \} \big) \nonumber \\
  &\quad= \sum\limits_{\ell} \mathbb{P}\big( \mathcal{H}_{k-1}^+ \cap \{ |{\bf S}_k|=\ell \} \big) \prod_{i=1}^{\ell} \mathbb{P}\left( H_{k,i} \,\, \left| \,\, \mathcal{H}_{k-1}^+ \cap \{ |{\bf S}_k|=\ell \} \cap \bigcap_{j=1}^{i-1} H_{k,j} \right. \right), \label{eq:firstProduct}
\end{align}
where the sum is over all integers $\ell$ such that the conditional probabilities above are well-defined. Intuitively, in the last step, we are treating the selection of the random vector ${\bf S}_k$ as a sequential selection of its components, which leads us to consider the product of suitable conditional probabilities. The next lemma provides a lower bound for the factors in this product.

\begin{lemma}\label{l:tedious}
For all $n$ large enough, we have
\[ \mathbb{P}\left( H_{k,i} \,\, \left| \,\, \mathcal{H}_{k-1}^+ \cap \{ |{\bf S}_k|=\ell \} \cap \bigcap_{j=1}^{i-1} H_{k,j} \right. \right) \geq \frac{\gamma}{2}, \]
for all $\ell\leq\xi_k$ and $i\leq\ell$ such that the conditional probability above is well-defined.
\end{lemma}

The idea of the proof of this lemma is that when a next component, $({\bf S}_k)_i$ is chosen, it is either a ``distinguished'' server, in the set $R_{k,i}$, or else it is a server chosen uniformly outside the set $R_{k,i}$ (cf.\ Proposition \ref{prop:symSampling}), in which case it has a substantial probability of being a busy server, in the set $B$. Although the intuition is clear, the formal argument is rather tedious and is deferred to Appendix \ref{app:inductionStep2}.\\

Applying Lemma \ref{l:tedious} to Equation \eqref{eq:firstProduct}, and using the fact that $\mathbb{P}\big( \mathcal{H}_{k}^- \big)\geq\alpha_k^->0$, we obtain
\begin{align*}
  \mathbb{P}(\mathcal{H}_{k}) &\geq \sum\limits_{\ell} \mathbb{P}\big( \mathcal{H}_{k-1}^+ \cap \{ |{\bf S}_k|=\ell \} \big) \left( \frac{\gamma}{2} \right)^\ell \\
  &\geq \mathbb{P}\big( \mathcal{H}_{k-1}^+ \cap  \{ |{\bf S}_k|\leq \xi_k \} \big) \left( \frac{\gamma}{2} \right)^{\xi_k}, \\
  &= \mathbb{P}\big( \mathcal{H}_{k}^- \big) \left( \frac{\gamma}{2} \right)^{\xi_k} \\
  &\geq \alpha_k^- \left( \frac{\gamma}{2} \right)^{\xi_k} \triangleq \alpha_k >0,
\end{align*}
for all $n$ large enough.
\end{proof}

\begin{lemma}\label{lem:induction_step3}
Suppose that $\mathbb{P}(\mathcal{H}_{k}) \geq \alpha_k$. Then, there exist a constant $\alpha_k^+>0$, such that for all $n$ large enough, we have $\mathbb{P}(\mathcal{H}_{k}^+) \geq \alpha_k^+$.
\end{lemma}
The proof is similar to the proof of Proposition \ref{prop:induction_step2} but with $\xi_k=1$, and it is  omitted. Intuitively, choosing the destination of a job has the same statistical properties as choosing one more server to sample, which brings us back to the setting of Proposition \ref{prop:induction_step2}. \\

This concludes the induction step. It follows that there exists a constant $\alpha_{c+1}^+>0$, which is independent of $n$, and such that $\mathbb{P}(\mathcal{H}_{c+1}^+) \geq \alpha_{c+1}^+$.

\subsection{Upper bound on the number of useful distinguished servers}\label{sec:boundUsefulServers}
Let us provide some intuition on what comes next. The dispatcher initially may treat in a non-typical manner the servers in an initial set of at most $c$ distinguished servers. As servers get sampled, the dispatcher acquires and possibly stores information about other servers. Ultimately, at the time of the $(c+1)$-st arrival, the dispatcher may have acquired information and therefore treat in a special manner (i.e., asymmetrically) the servers in the set
\begin{equation}\label{eq:definitionR}
 \overline{R}\triangleq \bigcup_{k=1}^{c+1}
 \left( R_k \cup R'_k \right),
\end{equation}
Recall that, for  $k=1,\dots,c+1$, we have
\begin{align}
 R_k = \bigcup_{i=1}^{|{\bf S}_k|} R_{k,i},
\end{align}
where each of the sets in the union has cardinality at most $c$, by Proposition \ref{prop:symSampling}. Furthermore, for $k=1,\dots,c+1$, the cardinality of $R'_k$ is also at most $c$, by Proposition \ref{prop:symDispatching}. It follows that
\begin{equation}\label{eq:RfirstBound}
  \big|\overline{R}\big|\leq c\sum_{k=1}^{c+1} \big( 1+|{\bf S}_k| \big).
\end{equation}
If we are to rely solely on this upper bound, the size of $\overline{R}$ can be larger than $c+1$, and it is possible in principle that the knowledge of so many ``distinguished'' servers (in the set $\overline{R}$) is enough for the dispatcher to identify $c+1$ idle servers to which to route the first $c+1$ jobs. On the other hand, under the event $\mathcal{H}_{c+1}^+$, all new information comes from servers that are ``busy'' (in the set $B$), and hence cannot be useful for the dispatching decisions. The next proposition states that for every sample path $\omega\in \mathcal{H}^+_{c+1}$, the set of idle (and therefore, potentially useful) servers on which information is available, namely, the set $\overline{R}\backslash B$, has cardinality of at most $c$.

\begin{proposition}\label{prop:cardinalityBound}
  The event $\mathcal{H}_{c+1}^+$ implies the event $\left| \overline{R}\backslash B \right| \leq c$.
\end{proposition}
\begin{proof}
Let us fix a realization $\omega\in\mathcal{H}^+_{c+1}$. We will upper bound the number of distinct images of the set $\overline{R}\backslash B$ under permutations of the set $\mathcal{N}_n$ of servers, which will lead to an upper bound on the cardinality of the set itself. In order to simplify notation, we will suppress the notational dependence on $\omega$ of all random variables for the rest of this proof.

We introduce a subset of the set of all possible permutations of $\mathcal{N}_n$,   with this subset being rich enough to lead to the desired bound. Towards this goal, we define the set
\begin{equation}
 F\triangleq \bigcup_{k=1}^{c+1} \left( \bigcup_{i=1}^{|{\bf S}_k|} \Big[\big\{({\bf S}_k)_i\big\}\backslash R_{k,i}\Big] \cup \Big[\{D_k\}\backslash \big(R'_k\cup {\bf S}_k^{set}\big)\Big] \right).
\end{equation}
This is the set of servers that were sampled, or that were chosen as the destination for a job, which were not in the distinguished sets $R_{k,i}$, or $R_k'\cup {\bf S}_k^{set}$, respectively.

Using our assumption $\omega\in \mathcal{H}_{c+1}^+$ and the definition of $\mathcal{H}_{c+1}^+$, we have
\[ \bigcup_{k=1}^{c+1} \bigcup_{i=1}^{|{\bf S}_k|} \big\{({\bf S}_k)_i\big\}\backslash R_{k,i} \subset B, \quad \text{and} \quad \bigcup_{k=1}^{c+1} \{D_k\} \backslash \big(R'_k\cup {\bf S}_k^{set}\big) \subset B. \]
As a result, we have $F\subset B$, and thus
\begin{equation}\label{eq:disjoint}
 \big(\overline{R}\backslash B\big) \cap F =\emptyset.
\end{equation}

Let $\Sigma$ be the set of permutations $\sigma$ of the server set $\mathcal{N}_n$ that:
\begin{itemize}
\item [(i)] preserve the ordering of $\overline{R}\backslash B$ in the sense defined in Section \ref{sec:notation},
\item [(ii)] fix the set $\big(\overline{R}\cap B\big)\cup F$, and
\item [(iii)] satisfy  $\sigma\big(\overline{R}\backslash B\big)\cap \big(\overline{R}\backslash B\big)=\emptyset$.
\end{itemize}
Consider two permutations $\sigma,\tau\in\Sigma$ such that $\sigma\big(\overline{R}\backslash B\big)=\tau\big(\overline{R}\backslash B\big)$. Then, the fact that $\sigma$ and $\tau$ both preserve the order of $\overline{R}\backslash B$ implies that $\sigma(i)=\tau(i)$, for all $i\in \overline{R}\backslash B$.

\begin{lemma}\label{lem:finalLemma}
  Let $\sigma,\tau\in\Sigma$, and let $\sigma_M$ and $\tau_M$, respectively, be associated permutations of the memory states as specified in Assumption \ref{def:symmetry} (Symmetry). Let $m(0)$ be the initial memory state, at time $0$. If $\sigma_M\big(m(0)\big)=\tau_M\big(m(0)\big)$, then $\sigma\big(\overline{R}\big) = \tau\big(\overline{R}\big)$.
\end{lemma}

Loosely speaking, Lemma \ref{lem:finalLemma} asserts that for the given sample path, permutations $\sigma,\tau$ in $\Sigma$ that lead to different sets $\overline{R}$ of distinguished servers must also lead (through $\sigma_M$ and $\tau_M$) to different initial memory states. The proof is an elementary consequence of our symmetry assumption on the underlying dynamics. However, it is tedious and is deferred to Appendix \ref{app:lastLemma}.\\

By Lemma \ref{lem:finalLemma}, and for $\sigma\in\Sigma$, distinct  images $\sigma(\overline{R})$  must correspond to distinct memory states $\sigma_M(m(0))$. Since the number of different memory states is upper bounded by $n^c$, this implies that
\begin{equation*}
  \Big|\big\{\sigma(\overline{R}) : \sigma \in\Sigma\big\}\Big| \leq n^c.
\end{equation*}
Furthermore, since every $\sigma\in\Sigma$ fixes the set $\overline{R}\cap B$, we have
\begin{equation}\label{eq:memoryBound}
  \Big|\big\{\sigma\big(\overline{R}\backslash B\big) : \sigma \in\Sigma\big\}\Big| = \Big|\big\{\sigma(\overline{R}) : \sigma \in \Sigma \big\}\Big| \leq n^c.
\end{equation}

Recall now that the only restrictions on the image $\sigma\big(\overline{R}\backslash B\big)$ under permutations in $\sigma\in\Sigma$ is that the set $\big(\overline{R}\cap B\big) \cup F$ is fixed, and that $\sigma\big(\overline{R}\backslash B\big)\cap \big(\overline{R}\backslash B\big)=\emptyset$. This implies that $\sigma\big(\overline{R}\backslash B\big)$ can be any set of the same cardinality within $\big(\overline{R}\cup F\big)^c$. It follows that
\begin{equation}\label{eq:permutationBound}
 \Big|\big\{\sigma\big(\overline{R}\backslash B\big) : \sigma \in\Sigma\big\}\Big| \geq {n - |\overline{R} \cup F| \choose |\overline{R}\backslash B|}.
\end{equation}
Recall also that under the event $\mathcal{H}_{c+1}^+$ we must have $\big|{\bf S}_k\big|\leq \xi_k$, for $k=1,\dots,c+1$. Thus $\big|F\big|\leq \xi_1+\dots+\xi_{c+1}+c+1\triangleq f$, and using Equation \eqref{eq:RfirstBound}, $|\overline{R}| \leq c(\xi_1+\dots+\xi_{c+1})+c+1 \triangleq  \theta$. Combining these two upper bounds, we obtain
\begin{equation*}
  {n - |\overline{R} \cup F| \choose |\overline{R}\backslash B|} \geq {n-(f+\theta) \choose |\overline{R}\backslash B|}.
\end{equation*}
Combining this with Equations \eqref{eq:memoryBound} and  \eqref{eq:permutationBound}, we obtain the inequality
\begin{equation}\label{eq:lastBound}
  n^c \geq {n-(f+\theta) \choose |\overline{R}\backslash B|}.
\end{equation}
Finally, using the bound $\big|\overline{R}\big|\leq \theta$, and applying Lemma \ref{lem:chooseBound}, we conclude that in order for this equation to hold for all $n$ large enough, we must have $\big|\overline{R}\backslash B\big|\leq c$.
\end{proof}

\subsection{Completing the proof} \label{s:compl}
We are now ready to complete the proof, by arguing that at least one of the first $c+1$  arrivals must be sent to a server that is either known to be busy or to a server on which no information is available, and therefore has positive probability of being busy.

Recall that for any fixed sample path in $\mathcal{H}_{c+1}^+$, we have (cf. Equation \eqref{eq:hkplus})
\[ \big\{D_1,\dots,D_{c+1}\big\} \subset B \cup  \bigcup_{k=1}^{c+1}
\big( {\bf S}_k^{set} \cup  R'_k \big). \]
Furthermore the event $\mathcal{H}_{c+1}^+$ implies that $({\bf S}_k)_i \in R_{k,i} \cup B$, for $i=1,\dots,|{\bf S}_k|$ and $k=1,\dots,c+1$. Therefore,
\begin{align}
 {\bf S}_k^{set} &\subset \bigcup_{i=1}^{|{\bf S}_k|} R_{k,i} \cup B = R_k \cup B,
\end{align}
for $k=1,\dots,c+1$. It follows that
\begin{align}
 \big\{D_1,\dots,D_{c+1}\big\} &\subset B \cup  \bigcup_{k=1}^{c+1} \big( {\bf S}_k^{set} \cup  R'_k \big) \\
 &\subset B \cup  \bigcup_{k=1}^{c+1} \big( R_k \cup R'_k \big) \\
 &= B \cup \overline{R}.
\end{align}
Moreover, Proposition \ref{prop:cardinalityBound} states that $\big| \overline{R} \backslash B \big| \leq c$. Thus, either (a) there exists $k$ such that $D_k\in B$, or (b) $D_i\in \overline{R}\backslash B$ for $i=1,\dots,c+1,$ and hence there exists a pair $k,l$, with $k<l$, such that $D_k=D_l$. We will now show that in both cases, the queueing delay is at least $\gamma$.

Let $L_k$ be the queueing delay of the $k$-th arrival. Recall that for $i\in B$, we have ${\bf Q}_{i,1}(0)>2\gamma$. Then, for case (a), with $D_k=i\in B$ we have
\begin{align*}
  L_k&=\big({\bf Q}_{i,1}(0)-T_k\big)^+ \geq 2\gamma - \frac{\gamma}{n} \geq \gamma > 0.
\end{align*}
On the other hand, for case (b), we have
\begin{align*}
L_l&\geq \big[W_k - (T_l - T_k)\big]^+ \geq 2\gamma - \left(\frac{\gamma}{n} - 0\right) \geq \gamma > 0.
\end{align*}
In both cases, we have
\[ L_1+\cdots+L_{c+1}\geq \gamma. \]
Since this is true for every sample path in $\mathcal{H}_{c+1}^+$, we obtain
\begin{equation}\label{eq:last}
 \mathbb{E}\big[L_1+\cdots+L_{c+1} \mid \mathcal{H}_{c+1}^+ \big]\geq \gamma.
\end{equation}

Finally, recall that the process $({\bf Q}(t),M(t),Z(t))_{t\geq 0}$ is stationary, with invariant probability measure $\pi_n$. Then, setting $t=\gamma/n$ in Equation \eqref{eq:delayDefinition}, we obtain
\begin{align*}
   \mathbb{E}_{\pi_n}^0\left[ L_0 \right] &= \frac{1}{\lambda \gamma} \mathbb{E}\left[ \sum\limits_{j=1}^{A_n\left(\frac{\gamma}{n}\right)} L_j \right] \\
   &\geq \frac{1}{\lambda \gamma} \mathbb{E}\left[\left. \sum\limits_{j=1}^{A_n\left(\frac{\gamma}{n}\right)} L_j \,\right|\, \mathcal{H}_{c+1}^+ \right] \mathbb{P}\big(\mathcal{H}_{c+1}^+\big) \\
   &\geq \frac{1}{\lambda \gamma} \mathbb{E}\left[\left. L_1+\cdots+L_{c+1} \,\right|\, \mathcal{H}_{c+1}^+ \right] \mathbb{P}\big(\mathcal{H}_{c+1}^+\big),
\end{align*}
where the last inequality comes from the fact that $\mathcal{H}_{c+1}^+\subset\big\{A_n(\gamma/n)\geq c+1\big\}$. Combining this with Equation \eqref{eq:last} and the fact that $\mathbb{P}\big(\mathcal{H}_{c+1}^+\big)\geq \alpha_{c+1}^+ > 0$, we obtain
\[ \mathbb{E}_{\pi_n}^0\left[ L_0 \right] \geq \frac{\alpha_{c+1}^+}{\lambda}>0. \]
As the constant in the lower bound does not depend on $n$, this completes the proof of the theorem.

\section{Conclusions and future work}\label{sec:conclusions}
We showed that when we have a limited amount of memory and a modest budget of messages per unit of time, and under a symmetry assumption, all dispatching policies result in queueing delay that is uniformly bounded away from zero. In particular, this implies that the queueing delay does not vanish as the system size increases.

Our result complements the results in \cite{positiveResult}, in which the authors showed that if we have a little more of either resource, i.e., if the number of memory bits or the message rate grows faster with $n$, then there exists a symmetric policy that drives the queueing delay to zero as $n\to\infty$. Consequently, we now have necessary and sufficient conditions on the amount of resources available to a central dispatcher, in order to achieve a vanishing queueing delay as the system size increases.

There are several interesting directions for future research. For example:
\begin{itemize}
\item [(i)] All the policies in the literature that achieve a vanishing queueing delay need a message rate at least equal to the arrival rate $\lambda n$. We conjecture that this is not a necessary condition for a policy to have a vanishing queueing delay, as long as it has access to the incoming job sizes.
\item [(ii)] We have focused on a system with homogeneous servers. For the case of nonhomogeneous servers, even stability can become an issue, and there are interesting tradeoffs between the resources used and the stability region. In this setting, we expect a result similar to our lower bound for queueing delay, stating that a resource constrained policy cannot be stable for every stabilizable system.
\end{itemize}

\appendix

\section{Comparison with a more restrictive symmetry assumption}
\label{app:restr}
In this appendix we explain why the stronger symmetry assumption
\begin{equation} \label{eq:appa}
\sigma\big(f_1(m,w,u)\big) = f_1\big(\sigma_M(m),w,u\big),\qquad \forall u\in[0,1],
\end{equation}
would be unduly restrictive.

Consider a policy that samples a fixed number $d$ of servers, uniformly at random (regardless of the memory state and of the incoming job size), and that satisfies this stronger symmetry assumption. Then, $f_1(m,w,u)$ is a vector of dimension $d$, for all $m\in\mathcal{M}_n$, $w\in\mathbb{R}_+$, and $u\in[0,1]$. Let $\sigma,\tau$ be a pair of permutations such that $\sigma(f_1(m,w,u))\neq \tau(f_1(m,w,u))$. The stronger symmetry assumption in Equation \eqref{eq:appa} implies that there exists a pair of associated permutations $\sigma_M,\tau_M$ of the memory states such that
\[ f_1\big(\sigma_M(m),w,u\big)=\sigma\big(f_1(m,w,u)\big)\neq \tau\big(f_1(m,w,u)\big)=f_1\big(\tau_M(m),w,u\big). \]
It follows that $\sigma_M(m)\neq \tau_M(m)$, and thus there must be at least as many memory states as the number of different vectors of dimension $d$ with different entries. There are ${n \choose d} d!$ such vectors, and therefore a large memory would be required to implement such a uniform sampling policy if Equation \ref{eq:appa} were to be enforced.

On the other hand, the symmetry assumption that we have adopted in this paper
 only requires equality in distribution, and uniform sampling can be achieved with only one memory state (i.e., with no bits of memory). Indeed, since the sampling of servers is done uniformly at random, we have
\[ f_1(m,w,U) \overset{d}{=} \sigma\big(f_1(m,w,U)\big), \]
for all permutations $\sigma$.

This example shows that the symmetry assumption that we have adopted can be substantially weaker (and thus easier to satisfy), and allows  small-memory implementation of simple natural policies.

\section{A combinatorial inequality } \label{app:delay}
We record here an elementary fact.

\begin{lemma}\label{lem:chooseBound}
Let us fix  positive integer constants $a$ and $c$. Suppose that $b$ satisfies
\begin{equation}\label{eq:l1}
{n-a \choose b} \leq n^c.
\end{equation}
As long as $n$ is large enough, we must have $b\leq c$ or $b\geq n-a-c$. \end{lemma}
\begin{proof}
Suppose that $b=c+1$. The quantity ${n-a \choose c+1}$ is a polynomial in $n$ of degree $c+1$ and therefore, when $n$ is large, \eqref{eq:l1} cannot hold.
In the range $c+1\leq b \leq (n-a)/2$, the quantity ${n-a \choose b}$ increases with $b$, and hence \eqref{eq:l1} cannot hold either.
Using the symmetry of the binomial coefficient, a similar argument is used to exclude the possibility that $(n-a)/2 \leq b \leq n-a-c-1$.
\end{proof}

\section{Proof of lemma 5.6}\label{app:inductionStep2}

In order to simplify notation, we introduce the following. For any $m\in\mathcal{M}_n$, $w\in\mathbb{R}_+$, and $b\in\mathcal{P}(\mathcal{N}_n)$, we define the event
\[ \mathcal{A}_{m,w,b}\triangleq \big\{ M(T_k^-)=m, B=b, W_k=w \big\}, \]
and we let $\mathbb{P}_{m,w,b}$ be the conditional probability measure
\[ \mathbb{P}_{m,w,b}(\,\cdot\,)\triangleq \mathbb{P}\big(\,\cdot \mid \mathcal{A}_{m,w,b}\big). \]

Let us fix some $\ell\leq\xi_k$ and some $i\leq\ell$. We have
\begin{align*}
  &\mathbb{P}\left( H_{k,i} \,\, \left| \,\, \mathcal{H}_{k-1}^+ \cap \{ |{\bf S}_k|=\ell \} \cap \bigcap_{j=1}^{i-1} H_{k,j} \right. \right) \\
  & \qquad = \int\limits_{m,w,b}  \mathbb{P}_{m,w,b}\left( H_{k,i} \,\left|\, \mathcal{H}_{k-1}^+ \cap \big\{ |{\bf S}_k|=\ell \big\} \cap \bigcap_{j=1}^{i-1} H_{k,j} \right. \right) \\
  &\qquad\qquad\qquad\qquad\qquad\quad \cdot \text{d}\mathbb{P}\left( \mathcal{A}_{m,w,b} \,\, \left| \,\, \mathcal{H}_{k-1}^+ \cap \{ |{\bf S}_k|=\ell \} \cap \bigcap_{j=1}^{i-1} H_{k,j} \right. \right).
\end{align*}
Moreover,
\begin{align*}
   &\mathbb{P}_{m,w,b}\left( H_{k,i} \,\left|\, \mathcal{H}_{k-1}^+ \cap \big\{ |{\bf S}_k|=\ell \big\} \cap \bigcap_{j=1}^{i-1} H_{k,j} \right. \right) \\
   &\qquad = \sum\limits_{{\bf s}} \mathbb{P}_{m,w,b}\left( H_{k,i} \,\left|\, \mathcal{H}_{k-1}^+ \cap \big\{ |{\bf S}_k|=\ell \big\} \cap \bigcap_{j=1}^{i-1} \{({\bf S}_k)_j={\bf s}_j \} \right. \right) \\
   &\qquad\qquad\qquad \cdot \mathbb{P}_{m,w,b}\left( \bigcap_{j=1}^{i-1}  \{({\bf S}_k)_j={\bf s}_j \} \,\left|\, \mathcal{H}_{k-1}^+ \cap \big\{ |{\bf S}_k|=\ell \big\} \cap \bigcap_{j=1}^{i-1} H_{k,j} \right. \right),
\end{align*}
where the sum is over all ($i-1$)-dimensional vectors $\bf s$ whose components are distinct indices of servers, and such that the conditional probabilities above are well-defined.

It is not hard to see that the desired result follows immediately once we establish the following claim.

\begin{claim}
 For all $n$ large enough, we have
\begin{equation}\label{eq:claim} \mathbb{P}_{m,w,b}\left( H_{k,i} \,\left|\, \mathcal{H}_{k-1}^+ \cap\, \big\{ |{\bf S}_k|=\ell \big\} \cap \bigcap_{j=1}^{i-1} \{ ({\bf S}_k)_j={\bf s}_j \} \right. \right) \geq  \frac{\gamma}{2},
\end{equation}
for all $(m,w,b,{\bf s})$ such that the conditional probability above is well-defined.
\end{claim}
\begin{proof}
Let us fix some $(m,w,b,{\bf s})$. Since $\mathcal{H}_{k-1}^+$ implies $|B|\geq \gamma n$, we have
\begin{equation}\label{eq:bigB}
  |b|\geq \gamma n.
\end{equation}
On the other hand, recall that
\[ H_{k,i} = \left\{ ({\bf S}_k)_i \in R_{k,i}\cup B \right\}, \]
where
\[ {\bf S}_k=f_1\big(M(T_k^-),W_k,U_k\big), \]
and $R_{k,i}$ is equal to the set
\[ R\Big(M(T_k^-),W_k,\big(({\bf S}_k)_1,\dots,({\bf S}_k)_{i-1}\big),|{\bf S}_k|\Big) \]
defined in Proposition \ref{prop:symSampling}, whenever the proposition applies. Otherwise, we have $R_{k,j}=\emptyset$. In any case, $R_{k,j}$ is a deterministic function of the same random variables. Then, conditioned on $M(T_k^-)=m$, $W_k=w$, $B=b$, $\big(({\bf S}_k)_1,\dots,({\bf S}_k)_{j-1}\big)={\bf s}$, and $|{\bf S}_k|=\ell$, we have
\[ H_{k,i} = \Big\{ \Big(f_1\big(m,w,U_k\big)\Big)_i \in r_{k,i}\cup b \Big\}, \]
where $r_{k,i}$ denotes the corresponding realization of the random set $R_{k,i}$. Note that the only randomness left in this event comes from $U_k$, which is a randomization random variable that is chosen independent from all the events prior to time $T_k^-$. It follows that $H_{k,i}$ is conditionally independent from $\mathcal{H}_{k-1}^+$, and thus
\begin{align*}
 &\mathbb{P}_{m,w,b}\left( H_{k,i} \,\left|\,  \mathcal{H}_{k-1}^+ \cap \big\{ |{\bf S}_k|=\ell \big\} \cap \bigcap_{j=1}^{i-1} \{ ({\bf S}_k)_j={\bf s}_j \} \right. \right) \nonumber \\
 &\qquad\qquad\qquad\qquad\qquad = \mathbb{P}_{m,w,b}\left( H_{k,i} \,\left|\, \big\{ |{\bf S}_k|=\ell \big\} \cap \bigcap_{j=1}^{i-1} \{ ({\bf S}_k)_j={\bf s}_j \} \right. \right).
\end{align*}

We now define the event
$G_{k,{\bf s},i,\ell}$ to be
\[ G_{k,{\bf s},i,\ell} \triangleq \big\{ |{\bf S}_k|=\ell \big\} \cap \bigcap_{j=1}^{i-1} \{ ({\bf S}_k)_j={\bf s}_j \}. \]
We are interested in bounding $\mathbb{P}_{m,w,b}\left( H_{k,i} \mid G_{k,{\bf s},i,\ell}\right)$, which we decompose into two terms:
\begin{align}
 \mathbb{P}_{m,w,b}\big( H_{k,i} \,\big|\, G_{k,{\bf s},i,\ell} \big)
 & = \mathbb{P}_{m,w,b}\big( ({\bf S}_k)_i \in r_{k,i}\cup b \,\big|\, G_{k,{\bf s},i,\ell} \big) \nonumber \\
 & = \mathbb{P}_{m,w,b}\big( ({\bf S}_k)_i \in r_{k,i} \,\big|\, G_{k,{\bf s},i,\ell} \big) \nonumber \\
 &\qquad\qquad\qquad\qquad\quad + \mathbb{P}_{m,w,b}\big( ({\bf S}_k)_i \in b\backslash r_{k,i} \,\big|\, G_{k,{\bf s},i,\ell} \big). \label{eq:split}
\end{align}
Since the conditional probability measure $\mathbb{P}_{m,w,b}(\,\cdot \mid G_{k,{\bf s},i,\ell})$ is well-defined, and since $\ell\leq\xi_k$ and $\xi_l\leq\sqrt{n}$ for all $n$ large enough, Proposition \ref{prop:symSampling} applies and yields
\begin{align}
&\mathbb{P}_{m,w,b}\big( ({\bf S}_k)_i = s \,\big|\, G_{k,{\bf s},i,\ell} \big)
= \mathbb{P}_{m,w,b}\big( ({\bf S}_k)_i = s' \,\big|\, G_{k,{\bf s},i,\ell} \big), \label{eq:equiProb}
\end{align}
for all $s,s'\notin r_{k,i} \cup \{{\bf s}_1,\dots,{\bf s}_{i-1}\}$. As a result,
\begin{align*}
  &\mathbb{P}_{m,w,b}\big( ({\bf S}_k)_i \in b\backslash r_{k,i} \,\big|\, G_{k,{\bf s},i,\ell} \big) \\
  &\quad\,\, \geq \mathbb{P}_{m,w,b}\big( ({\bf S}_k)_i \in b\backslash (r_{k,i}\cup \{{\bf s}_1,\dots,{\bf s}_{i-1}\} ) \,\big|\, G_{k,{\bf s},i,\ell} \big) \\
  &\quad\,\, =\frac{|b\backslash (r_{k,i} \cup \{{\bf s}_1,\dots,{\bf s}_{i-1}\})|}{n-|r_{k,i} \cup \{{\bf s}_1,\dots,{\bf s}_{i-1}\}|} \mathbb{P}_{m,w,b}\big( ({\bf S}_k)_i \notin r_{k,i}\cup \{{\bf s}_1,\dots,{\bf s}_{i-1}\} \,\big|\, G_{k,{\bf s},i,\ell} \big).
\end{align*}
Moreover, using the facts that $|b|\geq \gamma n$ (Equation \eqref{eq:bigB}), $|r_{k,i}|\leq c$ (Proposition \ref{prop:symSampling}), and $i\leq \ell$, we obtain
\begin{align*}
  &\frac{|b\backslash (r_{k,i} \cup \{{\bf s}_1,\dots,{\bf s}_{i-1}\})|}{n-|r_{k,i} \cup \{{\bf s}_1,\dots,{\bf s}_{i-1}\}|} \,\cdot\,\mathbb{P}_{m,w,b}\big( ({\bf S}_k)_i \notin r_{k,i}\cup \{{\bf s}_1,\dots,{\bf s}_{i-1}\} \,\big|\, G_{k,{\bf s},i,\ell} \big) \\
  &\qquad\qquad\qquad\quad \geq \frac{\gamma n - c - \ell}{n}\,\cdot\, \mathbb{P}_{m,w,b}\big( ({\bf S}_k)_i \notin r_{k,i}\cup \{{\bf s}_1,\dots,{\bf s}_{i-1}\} \,\big|\, G_{k,{\bf s},i,\ell} \big) \\
  &\qquad\qquad\qquad\quad \geq \frac{\gamma}{2} \,\cdot\,\mathbb{P}_{m,w,b}\big( ({\bf S}_k)_i \notin r_{k,i}\cup \{{\bf s}_1,\dots,{\bf s}_{i-1}\} \,\big|\, G_{k,{\bf s},i,\ell} \big),
\end{align*}
when $n$ is large enough. Finally, since the elements of the vector ${\bf S}_k$ are distinct,
\begin{align*}
  &\mathbb{P}_{m,w,b}\big( ({\bf S}_k)_i \notin r_{k,i}\cup \{{\bf s}_1,\dots,{\bf s}_{i-1}\} \,\big|\, G_{k,{\bf s},i,\ell} \big) = \mathbb{P}_{m,w,b}\big( ({\bf S}_k)_i \notin r_{k,i} \,\big|\, G_{k,{\bf s},i,\ell} \big),
\end{align*}
and therefore
\[ \mathbb{P}_{m,w,b}\big( ({\bf S}_k)_i \in b\backslash r_{k,i} \,\big|\, G_{k,{\bf s},i,\ell} \big) \geq \frac{\gamma}{2}\mathbb{P}_{m,w,b}\big( ({\bf S}_k)_i \notin r_{k,i} \,\big|\, G_{k,{\bf s},i,\ell} \big). \]
We now substitute into Equation \eqref{eq:split}, and obtain
\begin{align*}
 &\mathbb{P}_{m,w,b}\big( H_{k,i} \,\big|\, G_{k,{\bf s},i,\ell} \big) \\
 &\qquad\qquad\quad\, \geq \mathbb{P}_{m,w,b}\big( ({\bf S}_k)_i \in r_{k,i} \,\big|\, G_{k,{\bf s},i,\ell} \big) + \frac{\gamma}{2} \mathbb{P}_{m,w,b}\big( ({\bf S}_k)_i \notin r_{k,i} \,\big|\, G_{k,{\bf s},i,\ell} \big) \\
 &\qquad\qquad\quad\, \geq \frac{\gamma}{2},
\end{align*}
for all $n$ large enough.
\end{proof}

\section{Proof of lemma 5.9} \label{app:lastLemma}

We first prove a claim about the set-valued functions $R$ and $R'$ introduced in Propositions \ref{prop:symSampling} and \ref{prop:symDispatching}, respectively.

\begin{claim}\label{claim:niceR}
For every $m\in\mathcal{M}_n$, $w\in\mathbb{R}_+$, ${\bf s}\in\mathcal{S}_n$ with $|{\bf s}|\leq\sqrt{n}$, ${\bf q}\in\mathcal{Q}^{|{\bf s}|}$, and for $\ell=|{\bf s}|+1,\dots,n$, and for every permutation $\sigma$, we have $R\big(\sigma_M(m),w,\sigma({\bf s}),\ell\big)=\sigma\big(R(m,w,{\bf s},\ell)\big)$, and $R'\big(\sigma_M(m),w,\sigma({\bf s}),{\bf q}\big)=\sigma\big(R'(m,w,{\bf s},{\bf q})\big)$.
\end{claim}
\begin{proof}
 In order to simplify notation, we suppress the dependence on $w$ of the functions $R$, $R'$, and $f_1$ throughout the proof of the lemma.

  Let $U$ be a uniform random variable over $[0,1]$. For every $m\in\mathcal{M}_n$, we define the random vector ${\bf S}(m)=f_1(m,U)$. Recall that $R\big(m,{\bf s},\ell\big)\subset \mathcal{N}_n\backslash {\bf s}^{set}$ is the unique set of minimal cardinality such that
  \begin{align*}
 &\mathbb{P}\left({\bf S}(m)_{|{\bf s}|+1} = j \ \left| \  \big\{ |{\bf S}(m)|=\ell \big\} \cap \bigcap_{i=1}^{|{\bf s}|} \big\{ {\bf S}(m)_i={\bf s}_i \big\} \right. \right) \nonumber \\
  &\qquad\qquad\qquad = \mathbb{P}\left({\bf S}\big(m\big)_{|{\bf s}|+1} = j' \ \left| \ \big\{ |{\bf S}(m)|=\ell \big\} \cap \bigcap_{i=1}^{|{\bf s}|} \big\{ {\bf S}(m)_i={\bf s}_i \big\} \right. \right),
 \end{align*}
for all $j,j' \notin R(m,{\bf s},\ell)\cup {\bf s}^{set}$. It is not hard to see, e.g., by replacing $j,j'$ in the above equality by $\sigma^{-1}(j),\sigma^{-1}(j')\notin R(m,{\bf s},\ell)\cup {\bf s}^{set}$, that
$\sigma\big(R\big(m,{\bf s},\ell\big)\big)\subset \mathcal{N}_n\backslash \sigma({\bf s}^{set})$ is the unique set of minimal cardinality such that
\begin{align*}
 &\mathbb{P}\left( \sigma\big({\bf S}(m)_{|{\bf s}|+1}\big) = j \ \left| \  \big\{ |\sigma({\bf S}(m))|=\ell \big\} \cap \bigcap_{i=1}^{|{\bf s}|} \big\{ \sigma({\bf S}(m)_i)=\sigma({\bf s}_i) \big\} \right. \right) \\
  &\quad = \mathbb{P}\left(\sigma\big({\bf S}(m)_{|{\bf s}|+1}\big) = j' \ \left| \ \big\{ |\sigma({\bf S}(m))|=\ell \big\} \cap \bigcap_{i=1}^{|{\bf s}|} \big\{ \sigma({\bf S}(m)_i)=\sigma({\bf s}_i) \big\} \right. \right),
 \end{align*}
for all $j,j' \notin \sigma\big(R\big(m,{\bf s},\ell\big)\big)\cup \sigma({\bf s}^{set})$. On the other hand, the symmetry assumption states that
\[ \sigma\big({\bf S}(m)\big)\overset{d}{=} {\bf S}\big(\sigma_M(m)\big).  \]
Combining the last two equalities we get that
$\sigma\big(R\big(m,{\bf s},\ell\big)\big)\subset \mathcal{N}_n\backslash \sigma({\bf s}^{set})$ is the unique set of minimal cardinality such that
  \begin{align*}
 &\mathbb{P}\left({\bf S}\big(\sigma_M(m)\big)_{|{\bf s}|+1} = j \ \left| \ \big\{ \big|{\bf S}\big(\sigma_M(m)\big)\big|=\ell \big\} \cap \bigcap_{i=1}^{|{\bf s}|} \big\{ {\bf S}\big(\sigma_M(m)\big)_i=\sigma({\bf s}_i) \big\} \right. \right) \\
  & = \mathbb{P}\left({\bf S}\big(\sigma_M(m)\big)_{|{\bf s}|+1} = j' \ \left| \ \big\{ \big|{\bf S}\big(\sigma_M(m)\big)\big|=\ell \big\} \cap \bigcap_{i=1}^{|{\bf s}|} \big\{ {\bf S}\big(\sigma_M(m)\big)_i=\sigma\big({\bf s}_i\big) \big\} \right. \right),
 \end{align*}
for all $i,j \notin \sigma\big(R\big(m,{\bf s},\ell\big)\big)\cup \sigma({\bf s}^{set})$. However, this is exactly the definition of $R(\sigma_M(m),\sigma({\bf s}),\ell)$ (uniqueness is crucial at this point), so we have
\begin{align*}
 \sigma\big(R\big(m,{\bf s},\ell\big)\big) &= R\big(\sigma_M(m),\sigma({\bf s}),\ell\big).
\end{align*}
The proof of $R'\big(\sigma_M(m),\sigma({\bf s}),{\bf q}\big)=\sigma\big(R'(m,{\bf s},{\bf q})\big)$ is analogous (this time making use of the symmetry of the mapping $f_2$) and is omitted.
\end{proof}

We continue with the proof of Lemma \ref{lem:finalLemma}. Under the event  $\mathcal{H}^+_{c+1}$, we  have $({\bf S}_1)_i\in R_{1,i}\cup B$, for $i=1,\dots,|{\bf S}_1|$. Applying Claim \ref{claim:niceR} and the fact  $m(t_1^-)=m(0)$, which implies that $\sigma_M\big(m(t_1^-)\big)=\tau_M\big(m(t_1^-)\big)$, we obtain
\begin{align}
 \sigma\big(R_{1,1}\big) &= \sigma\Big(R\big(m(t_1^-),w_1,\emptyset,|{\bf S}_1|\big)\Big) \\
 &= R\Big(\sigma_M\big(m(t_1^-)\big),w_1,\emptyset,|{\bf S}_1|\Big) \\
 &= R\Big(\tau_M\big(m(t_1^-)\big),w_1,\emptyset,|{\bf S}_1|\Big) \\
 &= \tau\Big(R\big(m(t_1^-),w_1,\emptyset,|{\bf S}_1|\big)\Big) \\
 &= \tau\big(R_{1,1}\big). \label{eq:sameR1}
\end{align}


Now recall that $\sigma$ and $\tau$ preserve the order of $\overline{R}\backslash B$ and fix $\overline{R}\cap B$, so in particular they preserve the order of $R_{1,1}\backslash B\subset \overline{R}\backslash B$ and fix $R_{1,1}\cap B\subset \overline{R}\cap B$. Combining this with Equation \eqref{eq:sameR1}, we must have $\sigma(i)=\tau(i)$, for all $i\in R_{1,1}$.
If $({\bf S}_1)_1\in R_{1,1}$, this implies that
\begin{equation}
\label{eq:perm}
 \tau\big(({\bf S}_1)_1\big)=\sigma\big(({\bf S}_1)_1\big).
 \end{equation}
On the other hand, if $({\bf S}_1)_1$ does not belong to $R_{1,1}$, then, from the definition of $F$, we must have $({\bf S}_1)_1\in F$. Since $\sigma$ and $\tau$ fix the set $F$, we conclude that Equation \eqref{eq:perm} must hold in all cases.

Proceeding inductively, and using the same argument, we obtain
\begin{equation}
\sigma\big(R_{1,i}\big)=\tau\big(R_{1,i}\big),
\end{equation}
for $i=1,\dots,|{\bf S}_1|$, and $\sigma(i)=\tau(i)$, for all $i\in {\bf S}_1^{set}$. It follows that $\sigma({\bf S}_1)=\tau({\bf S}_1)$. Combining this with the fact that $\sigma_M\big(m(t_1^-)\big)=\tau_M\big(m(t_1^-)\big)$, and applying Claim \ref{claim:niceR} twice, we obtain
\begin{align}
 \sigma\Big(R'_{1}\Big) &= \sigma\Big(R'\big(m(t_1^-),w_1,{\bf S}_1,{\bf q}_{{\bf S}_1}(t_1^-)\big)\Big) \\
 &= R'\Big(\sigma_M\big(m(t_1^-)\big),w_1,\sigma\big({\bf S}_1\big),{\bf q}_{{\bf S}_1}(t_1^-)\Big) \\
 &= R'\Big(\tau_M\big(m(t_1^-)\big),w_1,\tau\big({\bf S}_1\big),{\bf q}_{{\bf S}_1}(t_1^-)\Big) \\
 &= \tau\Big(R'\big(m(t_1^-),w_1,{\bf S}_1,{\bf q}_{{\bf S}_1}(t_1^-)\big)\Big) \\
 &= \tau\Big(R'_{1}\Big). \label{eq:sameR'1}
\end{align}

Now recall that $\sigma$ and $\tau$ preserve the order of $\overline{R}\backslash B$ and fix $\overline{R}\cap B$,  so in particular they preserve the order of $R'_1\backslash B\subset \overline{R}\backslash B$ and fix $R'_1\cap B \subset \overline{R}\cap B$. Combining this with Equation \eqref{eq:sameR'1}, we must have $\sigma(i)=\tau(i)$, for all $i\in R'_1$. Furthermore, recall that we also have that $\sigma(i)=\tau(i)$, for all $i\in {\bf S}_1^{set}$. If $D_1\in R'_1\cup {\bf S}_1^{set}$, this implies that
\begin{equation}
  \sigma\big(D_1\big) = \tau\big(D_1\big). \label{eq:sameD1}
\end{equation}
On the other hand, if $D_1$ does not belong to $R'_1\cup {\bf S}_1^{set}$, then, from the definition of $F$, we must have $D_1\in F$. Since $\sigma$ and $\tau$ fix the set $F$, we conclude that Equation \eqref{eq:sameD1} must hold in all cases.

We now consider a memory update. Using the symmetry assumption, we have
  \begin{align*}
  \sigma_M\big(m(t_1)\big) &= \sigma_M\Big(f_3\big(m(t_1^-),w_1,{\bf S}_1,{\bf q}_{{\bf S}_1}(t_1^-),D_1\big)\Big) \\
  &= f_3\Big(\sigma_M\big(m(t_1^-)\big),w_1,\sigma\big({\bf S}_1\big),{\bf q}_{{\bf S}_1}(t_1^-),\sigma\big(D_1\big)\Big).
  \end{align*}
  Then, since $\sigma_M(m(t_1^-))=\tau_M(m(t_1^-))$, $\sigma\big({\bf S}_1\big) = \tau\big({\bf S}_1\big)$, and $\tau\big(D_1\big)=\sigma\big(D_1\big)$, we have
  \begin{align*}
  &f_3\Big(\sigma_M\big(m(t_1^-)\big),w_1,\sigma\big({\bf S}_1\big),{\bf q}_{{\bf S}_1}(t_1^-),\sigma\big(D_1\big)\Big) \\
  &\qquad\qquad\qquad\qquad\qquad\qquad =f_3\Big(\tau_M\big(m(t_1^-)\big),w_1,\tau\big({\bf S}_1\big),{\bf q}_{{\bf S}_1}(t_1^-),\tau\big(D_1\big)\Big).
  \end{align*}
  Using the symmetry assumption once again, we obtain
  \begin{align*}
  &f_3\Big(\tau_M\big(m(t_1^-)\big),w_1,\tau\big({\bf S}_1\big),{\bf q}_{{\bf S}_1}(t_1^-),\tau\big(D_1\big)\Big) \\
  & \qquad\qquad\qquad\qquad\qquad\qquad\qquad\quad = \tau_M\Big(f_3\big(m(t_1^-),w_1,{\bf S}_1,{\bf q}_{{\bf S}_1}(t_1^-),D_1\big)\Big) \\
  & \qquad\qquad\qquad\qquad\qquad\qquad\qquad\quad = \tau_M\big(m(t_1)\big).
  \end{align*}
  We conclude that
  \[ \sigma_M\big(m(t_1)\big)=\tau_M\big(m(t_1)\big). \]

Finally, since the memory states at time $t_1$ are still equal, we can proceed inductively by applying the same argument to obtain that, for $k=2,\dots,c+1$, we have $\sigma\big(R_{k,i}\big)=\tau\big(R_{k,i}\big)$ for $i=1,\dots,|{\bf S}_k|$, and $\sigma\big(R'_k\big)=\tau\big(R'_k\big)$. It follows that $\sigma\big(\overline{R}\big) = \tau\big(\overline{R}\big)$.


\bibliographystyle{imsart-number}
\bibliography{references}
\end{document}